\documentclass[english,12pt,oneside]{amsproc}
\usepackage[english]{babel}
\usepackage{a4wide}
\usepackage{amsthm}
\usepackage{graphics}
\usepackage{amsfonts, amssymb, amscd, amsmath}
\usepackage{latexsym}
\usepackage[matrix,arrow,curve]{xy}
\usepackage{mathabx}
\usepackage{color}
\usepackage{pbox}
\usepackage{tikz}
\usetikzlibrary{matrix,decorations.pathreplacing,positioning}
\usepackage{hyperref}

\DeclareMathOperator{\Cone}{Cone} 
\DeclareMathOperator{\Ker}{Ker}

\DeclareMathOperator{\Hom}{Hom} \DeclareMathOperator{\rk}{rk}
\DeclareMathOperator{\odd}{odd} \DeclareMathOperator{\SU}{SU}
\DeclareMathOperator{\Imm}{Im} \DeclareMathOperator{\com}{compl}
\DeclareMathOperator{\ab}{ab} \DeclareMathOperator{\ABb}{AB} \DeclareMathOperator{\link}{link}
\DeclareMathOperator{\relint}{relint} \DeclareMathOperator{\Tot}{Tot}
\DeclareMathOperator{\Hilb}{Hilb}
\DeclareMathOperator{\ord}{ord}

\newcommand{\inc}[2]{{[#1\!:\!#2]}}

\newcommand{\ko}{\Bbbk}
\newcommand{\Zo}{\mathbb{Z}}
\newcommand{\Ro}{\mathbb{R}}
\newcommand{\Rg}{\mathbb{R}_{\geqslant 0}}
\newcommand{\Co}{\mathbb{C}}
\newcommand{\Qo}{\mathbb{Q}}

\newcommand{\af}{a.f.}

\newcommand{\ttt}{\mathfrak{t}}

\newcommand{\AB}{\mathcal{AB}}

\newcommand{\ca}[1]{\mathcal{#1}}
\newcommand{\Hr}{\widetilde{H}}
\newcommand{\dd}{\partial}
\newcommand{\Ca}{\mathcal{C}}
\newcommand{\F}{\mathcal{F}}
\newcommand{\Hh}{\mathcal{H}}

\newcommand{\CP}{\mathbb{C}P}
\newcommand{\HP}{\mathbb{H}P}

\newcounter{stmcounter}[section]
\newcounter{thcounter}
\newcounter{problcounter}

\numberwithin{equation}{section}

\theoremstyle{plain}
\newtheorem{cor}[stmcounter]{Corollary}

\newtheorem{thm}[thcounter]{Theorem}

\newtheorem{prop}[stmcounter]{Proposition}
\newtheorem{lem}[stmcounter]{Lemma}
\newtheorem{probl}[problcounter]{Problem}

\theoremstyle{definition}
\newtheorem{defin}[stmcounter]{Definition}

\theoremstyle{remark}
\newtheorem{ex}[stmcounter]{Example}
\newtheorem{rem}[stmcounter]{Remark}
\newtheorem{con}[stmcounter]{Construction}

\begin{document}

\title[Equivariantly formal torus actions of complexity one]{Orbit spaces of equivariantly formal torus actions of complexity one}

\author{Anton Ayzenberg}
\address{Faculty of computer science, National Research University Higher School of Economics, Russian Federation, and Neapolis University Pafos, Paphos, Cyprus}
\email{ayzenberga@gmail.com}

\author{Mikiya Masuda}
\address{Osaka City University Advanced Mathematical Institute, Japan, and Faculty of computer science, National Research University Higher School of Economics, Russian Federation}
\email{masuda@sci.osaka-cu.ac.jp }

\date{\today}
\thanks{The article was prepared within the framework of the HSE University Basic Research Program}




\begin{abstract}
Let a compact torus $T=T^{n-1}$ act on an orientable smooth compact manifold $X=X^{2n}$ effectively, with nonempty finite set of fixed points, and suppose that stabilizers of all points are connected. If $H^{\odd}(X)=0$ and the weights of tangent representation at each fixed point are in general position, we prove that the orbit space $Q=X/T$ is a homology $(n+1)$-sphere. If, in addition, $\pi_1(X)=0$, then $Q$ is homeomorphic to $S^{n+1}$. We introduce the notion of $j$-generality of tangent weights of torus action. For any action of $T^k$ on $X^{2n}$ with isolated fixed points and $H^{\odd}(X)=0$, we prove that $j$-generality of weights implies $(j+1)$-acyclicity of the orbit space $Q$. This statement generalizes several known results for actions of complexity zero and one. In complexity one, we give a criterion of equivariant formality in terms of the orbit space. In this case, we give a formula expressing Betti numbers of a manifold in terms of certain combinatorial structure that sits in the orbit space.
\end{abstract}

\maketitle

\section{Introduction}\label{secIntro}

Let a compact torus $T=T^k$ act smoothly and effectively on an orientable connected closed smooth manifold $X=X^{2n}$ with nonempty finite set of fixed points. The number $n-k$ can be shown to be nonnegative. This number is called the complexity of the action.

For a fixed point $x\in X^T$ of the action, consider $\alpha_{x,1},\ldots,\alpha_{x,n}\in \Hom(T^k,T^1)\cong \Zo^k$, the weights of the tangent representation at $x$ defined up to sign.

\begin{defin}
The action is said to be in \emph{$j$-general position} if $j\leqslant n$ and, for any fixed point $x$, any $j$ of the weights $\alpha_{x,1},\ldots,\alpha_{x,n}$ are linearly independent over $\Qo$.
\end{defin}

We will usually assume that $j\leqslant k$ where $k$ is the dimension of the acting torus, since otherwise the condition is empty. An action of $T=T^k$ is called \emph{an action in general position}, if it is in $k$-general position.



For an action of $T^k$ on $X$ consider the equivariant filtration:
\[
X_0\subset X_1 \subset X_2\subset\cdots\subset X_k=X,
\]
where $X_i$ consists of torus orbits of dimension at most $i$. There is an orbit type filtration on the orbit space $Q=X/T$:
\[
Q_0\subset Q_1 \subset Q_2\subset\cdots\subset Q_k=Q,\qquad Q_i=X_i/T.
\]
In the following it is assumed that filtrations start with $X_{-1}=Q_{-1}=\varnothing$. Using the filtration of $Q$, one can define an $i$-dimensional face $F$ of $Q$ as a closure of any connected component of $Q_i\setminus Q_{i-1}$. If $p\colon X\to Q$ denotes the natural projection to the orbit space, then the full preimage $X_F=p^{-1}(F)$ of a face $F$ is a smooth submanifold of $X$, preserved by the $T$-action. We call $X_F$ a face submanifold of $X$. Let $T_F\subset T^k$ denote the noneffective kernel of the $T^k$-action on $X_F$. Therefore there is an effective torus action of $T/T_F$ on a face submanifold $X_F$.

We now briefly recall the notion of an equivariantly formal action. For a smooth action of the torus $T$ on an orientable smooth manifold $X$ consider the fibration $X\hookrightarrow X\times_TET\to BT$ and the corresponding Serre spectral sequence
\begin{equation}\label{eqSerreSeq}
E_2^{*,*}\cong H^*(BT)\otimes H^*(X)\Rightarrow H^*(X\times_TET)=H^*_T(X),
\end{equation}
where $T\hookrightarrow ET\to BT$ is the universal principal $T$-bundle, $X\times_TET$ is the Borel construction of $X$, and $H^*_T(X)$ is the equivariant cohomology algebra. The coefficients of cohomology modules are taken in the ring $R$, which is either $\Zo$ or $\Qo$. It will be assumed throughout the paper that either $R=\Zo$ and all stabilizers of the action are connected, or, otherwise, $R=\Qo$. The space $X$ with a torus action is called \emph{(cohomologically) equivariantly formal} in the sense of Goresky--Kottwitz--MacPherson \cite{GKM} if its Serre spectral sequence \eqref{eqSerreSeq} degenerates at $E_2$. In particular, the spaces with vanishing odd degree cohomology are all equivariantly formal. In \cite[Lm.2.1]{MasPan} it was proved that under the assumption that fixed points are isolated, the condition $H^{\odd}(X)=0$ is equivalent to $H^*_T(X)$ being a free module over $H^*(BT)$. For orientable manifolds, the latter condition is equivalent to equivariant formality according to~\cite[Thm.1.1]{FP} and~\cite[Cor.1.4]{Franz}.

According to~\cite[Lem.2.2]{MasPan}, the following condition holds for equivariantly formal spaces with isolated fixed points
\begin{equation}\label{eqCondJoint}
\mbox{each face of }Q\mbox{ has a vertex.}
\end{equation}
In the first part of the paper we study equivariantly formal actions, hence Condition \eqref{eqCondJoint} is satisfied. In the second part, condition~\eqref{eqCondJoint} is taken as a standing assumption.

The actions of complexity zero and their orbit spaces are well studied in toric topology. The following lemma describes the orbit spaces of equivariantly formal torus actions of complexity zero.

\begin{lem}[\cite{MasPan}]\label{lemMasPan}
Consider an equivariantly formal smooth effective action of $T=T^n$ on an orientable closed connected manifold $M=M^{2n}$ with isolated fixed points. Then the orbit space $P=M/T$ is acyclic. Its face structure given by the torus action is the structure of a homology cell complex on $P$.
\end{lem}

Consider an action of complexity one, that is an action of $T^{n-1}$ on $X^{2n}$, and assume that it is in general position. Then, according to \cite[Thm.2.10]{AyzCompl}, the orbit space $Q=X/T$ is a closed topological manifold of dimension $n+1$ provided that \eqref{eqCondJoint} holds true. Moreover, in this case $\dim Q_i=i$ and $\dim X_i=2i$ for $i\leqslant n-2$, while $\dim Q=n+1$. Therefore, $Q$ has faces of dimensions $0$ to $n-2$ and the unique maximal cell of dimension $n+1$ (which is the orbit manifold $Q$ itself). All faces except $Q$ will be called \emph{proper faces}. The set $Q_{n-2}$, which is the union of all proper faces, has specific topology which was axiomatized in the notion of \emph{a sponge} in \cite{AyzCompl}. Our main result concerning torus actions of complexity one is the following.

\begin{thm}\label{thmEqFormImpliesSphere}
Assume that a smooth action of $T^{n-1}$ on an orientable connected closed manifold $X=X^{2n}$ is effective, has nonempty finite set of fixed points, and is in general position. If the action is equivariantly formal, then the following hold for the orbit space $Q=X^{2n}/T^{n-1}$:
\begin{enumerate}
  \item all proper faces $F$ of $Q$ are acyclic, i.e. $\Hr^*(F)=0$;
  \item the sponge $Q_{n-2}$ is $(n-3)$-acyclic, i.e. $\Hr^i(Q_{n-2})=0$ for $i\leqslant n-3$;
  \item the orbit space $Q$ is a homology $(n+1)$-sphere, that is $\Hr^i(Q)=0$ for $i\leqslant n$ and $H^{n+1}(Q)\cong R$.
\end{enumerate}
Here, the coefficient ring $R$ is either $\Zo$ and all stabilizers of the action are assumed connected, or $R=\Qo$.
\end{thm}

\begin{cor}\label{corTopSphere}
Let $X$ be as in Theorem \ref{thmEqFormImpliesSphere} (for $\Zo$ coefficients), and, moreover, $X$ is simply connected. Then the orbit space $Q=X/T$ is homeomorphic to $S^{n+1}$.
\end{cor}

\begin{proof}
Since the set of fixed points is nonempty, the condition $\pi_1(X)=0$ implies $\pi_1(Q)=0$ by \cite[Corollary 6.3]{Bred}. Generalized Poincar\'{e} Conjecture in topological category \cite{SmalePC}, \cite[Sec.21.6.2]{PoinFour}, \cite[Sec.3.2]{NotesPerel} then implies the homeomorphism $Q\cong S^{n+1}$.
\end{proof}

Using Corollary \ref{corTopSphere} we recover the following particular results:
\begin{enumerate}
  \item The result of Buchstaber--Terzic \cite{BT,BT2} who initiated the study of actions of positive complexity in terms of their orbit spaces. Their result states the homeomorphisms $G_{4,2}/T^3\cong S^5$ and $F_3/T^2\cong S^4$ for the complex Grassmann manifold $G_{4,2}$ of $2$-planes in $\Co^4$ and the manifold $F_3$ of full complex flags in $\Co^3$.
  \item The result of Karshon--Tolman \cite{KTmain} which states that the orbit space of a Hamiltonian complexity one torus action in general position is homeomorphic to a sphere. Indeed, if $X$ has Hamiltonian action with isolated fixed points, then $\pi_1(X)=0$, see e.g. \cite{HuiLi}.
  \item The results of the first author \cite{AyzCompl,AyzHP} asserting that $X^{2n}/T^{n-1}$ is a sphere, provided that $X^{2n}$ is an equivariantly formal manifold with an action of $T^n$, the orbit space $X^{2n}/T^n$ is a disk, and $T^{n-1}\subset T^n$ is a subtorus in general position. This class of examples includes the classical action of a maximal torus $T^2\subset G_2$ on the 6-sphere $G_2/\SU(3)$ of unit imaginary octonions.
  \item The result of \cite{AyzHP} which asserts that $\HP^2/T^3\cong S^5$ and its generalization to other quaternionic toric manifolds of dimension $8$.
\end{enumerate}

We mention here that the same technique, as used in the proof of Theorem~\ref{thmEqFormImpliesSphere}, gives a new proof of Lemma~\ref{lemMasPan}. We prove Lemma~\ref{lemMasPan} in Section~\ref{secProofMain} before giving the proof of Theorem~\ref{thmEqFormImpliesSphere} to demonstrate the key ideas. Next, using similar arguments we prove

\begin{thm}\label{thmOrbitAcyclicityAnyComplexity}
Assume that a smooth action of $T=T^k$ on an orientable connected closed manifold $X=X^{2n}$ with isolated fixed points is equivariantly formal and in $j$-general position, $j\geqslant 1$. Then the orbit space $Q=X/T$ satisfies
\[
\Hr^i(Q)=0 \mbox{ for } i\leqslant j+1.
\]
It is assumed that either the action has connected stabilizers and the coefficients are taken in $\Zo$, or the coefficients are in $\Qo$.
\end{thm}

Both Lemma~\ref{lemMasPan} and Theorem~\ref{thmEqFormImpliesSphere} arise as particular cases of this theorem. Another particular case is the following

\begin{cor}
Assume that $X$ is a closed orientable GKM-manifold. Then its orbit space $Q=X/T$ is $3$-acyclic, that is $\Hr^i(Q)=0$ for $i=0,1,2,3$.
\end{cor}

Indeed, a GKM-manifold is an equivariantly formal manifold such that, at any fixed point, any two weights are linearly independent. This means that the torus action on a GKM-manifold is in 2-general position and Theorem~\ref{thmOrbitAcyclicityAnyComplexity} applies.

In Sections~\ref{secSphereImpliesFormalitySimple}, \ref{secSpongesHomology}, and~\ref{secSphereImpliesFormality} we prove the result, which is converse to Theorem~\ref{thmEqFormImpliesSphere}. Under certain assumptions on the complexity one action in general position, the acyclicity of all proper faces $F$, the $(n-3)$-acyclicity of the sponge $Q_{n-2}$, and the $n$-acyclicity of the orbit space $Q$ imply that the action is equivariantly formal. This statement is more complicated than the direct theorem. In Section~\ref{secSphereImpliesFormalitySimple} we formulate the theorem and demonstrate the key ideas for the case $n=2$, that is for circle actions on 4-dimensional manifolds, which is the classical subject of equivariant topology.

In order to formulate and prove the general statement, we recall the notion of a sponge from~\cite{AyzCompl} and prove several homological statements about sponges in Section~\ref{secSpongesHomology}, in particular, we review the necessary notions about (co)sheaves on finite posets, and recall the construction of the dihomology spectral sequence, which plays an important role in the arguments. In Section~\ref{secSphereImpliesFormality}, we formulate and prove the general theorem (Theorem~\ref{thmSphereImpliesEqForm}), converse to Theorem~\ref{thmEqFormImpliesSphere}. The proof requires additional homological machinery, the sheaf of Atiyah--Bredon complexes being the main ingredient.

Our criterion is related to the result of Franz~\cite{Franz}, however there is one difference. In~\cite{Franz}, the criterion of equivariant formality is stated for locally standard actions, and the first step of this technique consists in making equivariant blow-ups so that a given action becomes locally standard. We work with a more restricted class of actions, however, in our considerations, we study the original space, not the blown-up space.

In Section~\ref{secBetti}, we express Betti numbers of equivariantly formal actions of complexity one in general position in terms of combinatorial and topological characteristics of the orbit space. We introduce the notion of \emph{the $h$-vector} of an abstract sponge. If a sponge is constructed from an equivariantly formal action on $X$, its $h$-vector encodes the Betti numbers of $X$ of even degrees. In general, we suppose that $h$-vectors of sponges have many common properties with the $h$-vectors of simple polytopes (Dehn--Sommerville relations, nonnegativity) and hope that there exists a combinatorial and algebraic theory of sponges parallel to the theory of Stanley--Reisner rings for simple polytopes and face rings of simplicial posets. Our optimistic belief is that there should exist a notion of a ``sponge algebra'', defined combinatorially in terms of generators and relations, and this notion will allow to describe equivariant cohomology algebras for equivariantly formal manifolds with complexity one actions in general position.

\section{Actions of complexity $1$ in general position}\label{secProofMain}

Let a torus $T=T^k$ act smoothly on a closed manifold $X=X^{2n}$. Our main tool is Atiyah--Bredon--Franz--Puppe sequence for equivariant cohomology:
\begin{multline}\label{eqABseq}
0\to H^*_T(X)\stackrel{i^*}{\to} H^*_T(X_0)\stackrel{\delta_0}{\to}
H^{*+1}_T(X_1,X_0)\stackrel{\delta_1}{\to}\cdots\\\cdots
\stackrel{\delta_{k-2}}{\to}H^{*+k-1}_T(X_{k-1},X_{k-2})\stackrel{\delta_{k-1}}{\to}H^{*+k}_T(X,X_{k-1})\to 0,
\end{multline}
where $\delta_i$ is the connecting homomorphism in the long exact sequence of equivariant cohomology of the triple $(X_{i+1},X_i,X_{i-1})$. It is known that whenever the action of $T^k$ on $X=X^{2n}$ is equivariantly formal, \eqref{eqABseq} is exact: Atiyah \cite{Atiyah} proved the analogous statement for K-theory, Bredon \cite{Bredon} proved it for cohomology with rational coefficients, and Franz--Puppe \cite{FP} obtained the result for cohomology over integers. However, in the latter case, it is required that all stabilizers of the torus action are connected (this requirement is a bit weaker in the paper \cite{FP}). Note that \eqref{eqABseq} is a sequence of $H^*(BT)$-modules. In the next construction we fix a notation to be used throughout the paper.

\begin{con}\label{conLowerRank}
Recall that $Q=X/T$ is the orbit space, and the spaces $Q_i=X_i/T$ form an orbit type filtration of $Q$. A closure of any connected component of $Q_i\setminus Q_{i-1}$ is called \emph{a face of rank $i$} of the action. The preimage of $F$ in $X$ is denoted $X_F$. Let $T_F\subset T$ denote the noneffective kernel of the induced $T$-action on $X_F$. The subspace $X_F$ is a connected component of the set $X^{T_F}$, therefore $X_F$ is a smooth submanifold, we call it \emph{the face submanifold} (of rank $i$) corresponding to $F$. If $X$ is orientable, then so is $X_F$.

The symbol $F_{-1}$ denotes the union of all proper subfaces of $F$. In other words, $F_{-1}=F\cap Q_{i-1}$, the intersection of $F$ with the lower strata of the action. Similarly, $(X_F)_{-1}$ denotes the preimage of $F_{-1}$ in $X$. The subset $F_{-1}$ can be thought as some sort of boundary of $F$. The whole orbit space can be obtained by inductively attaching faces $F$ along subsets $F_{-1}$. It should be noted however, that $F_{-1}$ may not coincide with the topological boundary of~$F$.
\end{con}

\begin{rem}
Recall that a topological pair $(B,C)$ is called \textit{a homology cell} (of rank $i$, over the coefficient ring $R$) if $\Hr_*(B;R)=0$ and $\Hr_j(B,C;R)=0$ for $j\neq i$ and $\Hr_i(B,C;R)\cong R$. Although it is not necessary for homological applications, we also require that $B$ is a CW-complex of dimension $i$, and $C$ is its subcomplex. \emph{A homology cell complex} is a topological space obtained by inductive attachment of homology cells (one can simply replace $(D^i,S^{i-1})$ with $i$-dimensional homology cells in the definition of a CW-complex).
\end{rem}

We now demonstrate our key ideas by proving Lemma~\ref{lemMasPan}.


\begin{proof}[Proof of Lemma~\ref{lemMasPan}]
We proceed by induction on $n$. The case $n=0$ is trivial. Assume that the statement holds for all $k<n$. Consider an equivariantly formal action of $T^n$ on $M^{2n}$. Let $P=M^{2n}/T^n$ denote the orbit space, $M_i$ : the equivariant $i$-skeleton of $M$, $\{P_i=M_i/T^n\}$ : the filtration of $P$, and $M_F$ : the face submanifold corresponding to a proper face $F$. According to \cite[Lm.2.2]{MasPan}, the manifold $M_F$ inherits the property of vanishing odd degree cohomology. Therefore, by induction hypothesis, a face $F$ is a homology disc for $\dim F<n$, and the pair $(F,\dd F)$ is a (co)homology cell. It can be shown directly, that for the actions of complexity zero, $F_{-1}$ coincides with the topological boundary $\dd F$ of a face $F$. Induction hypothesis implies that the face stratification of $P_{n-1}$ is a homology cell complex.

Now we write the ABFP-sequence \eqref{eqABseq} for $M$
\begin{multline}\label{eqABseqForM}
0\to H^*_T(M)\stackrel{i^*}{\to} H^*_T(M_0)\stackrel{\delta_0}{\to}
H^{*+1}_T(M_1,M_0)\stackrel{\delta_1}{\to}\cdots\\\cdots
\stackrel{\delta_{n-2}}{\to}H^{*+n-1}_T(M_{n-1},M_{n-2})\stackrel{\delta_{n-1}}{\to}H^{*+n}_T(M,M_{n-1})\to 0.
\end{multline}
Consider the $i$-th term in \eqref{eqABseqForM} for $i\leqslant n-1$. We have
\begin{equation}\label{eqRelativeSum}
H^*_T(M_{i},M_{i-1})\cong \bigoplus_{F\colon\dim F=i}H^*_T(M_F,M_F\cap M_{i-1})\cong \bigoplus_{F\colon\dim F=i}H^{*}(F,F_{-1})\otimes H^*(BT_F)
\end{equation}
since the action of $T/T_F$ is (almost) free\footnote{An almost free action is an action with finite stabilizers. This situation may occur if disconnected stabilizers are allowed for the original action, however in this case we take coefficients in $\Qo$. For an almost free action of $T$ on $X$ we have $H^*_T(X;\Qo)\cong H^*(X/T;\Qo)$.} on $M_F\setminus M_{i-1}$. If $\dim F<n$, then $F$ is a homology cell by induction hypothesis, therefore $H^{j}(F,F_{-1})=H^{j}(F,\dd F)=0$ for $j<\dim F$. Therefore $H^{*+i}_T(M_{i},M_{i-1})=0$ if $i<n$ and $*<0$. Taking $*<0$ in \eqref{eqABseqForM} we get the exact sequence
\begin{equation}\label{eqRelCohomM}
0\to 0\to\cdots \to 0\stackrel{\delta_{n-1}}{\to}H^{*+n}_T(M,M_{n-1})\to 0
\end{equation}
which implies that $H^{*+n}_T(M,M_{n-1})=0$ for $*<0$. However, $H^{*+n}_T(M,M_{n-1})\cong H^{*+n}(P,P_{n-1})$, hence we get the $(n-1)$-acyclicity of the pair $(P,P_{n-1})$ in cohomology.

We have $\dim P_i=i$ since $\dim T=n$. Consider the cohomological spectral sequence associated with the filtration $P_0\subset\cdots\subset P_{n-1}\subset P_n=P$:
\[
E_2^{p,q}\cong H^{p+q}(P_p,P_{p-1})\Rightarrow H^{p+q}(P).
\]
We have $E_2^{p,q}=0$ for $q>0$ by dimensional reasons. We also have $E_2^{p,q}=0$ for $q<0$: for $p=n$ this is the statement of the previous paragraph, and for $p<n$ this follows from $P_{n-1}$ being a homology cell complex. Hence we only have one nontrivial row $E_2^{p,0}$.

Let us look at the components of the exact sequence \eqref{eqABseqForM} of degree $0$. We get the exactness of the sequence
\begin{equation}\label{eqPexSeq}
0\to R\to H^0(P_0)\to H^1(P_1,P_0)\to\cdots
\to H^{n-1}(P_{n-1},P_{n-2})\to H^n(P,P_{n-1})\to 0.
\end{equation}
The differentials in~\ref{eqPexSeq} and~\eqref{eqABseqForM} agree, since they are the connecting homomorphisms of the triples $(P_{i+1},P_i,P_{i-1})$ and $(M_{i+1},M_i,M_{i-1})$ respectively, so they commute with the isomorphisms~\eqref{eqRelativeSum}. The sequence~\eqref{eqPexSeq} is exactly the nontrivial row of the first page of the spectral sequence $E_2^{*,*}$. Therefore, acyclicity of $\eqref{eqPexSeq}$ implies that $H^i(P)=0$ for $i>0$.

Since $M$ is orientable, the orbit space $P$ is an orientable homology manifold with the boundary $P_{n-1}$ (see details in Lemma~\ref{lemOrientability} below). Poincare--Lefschetz duality implies that the pair $(P,P_{n-1})$ has the same relative homology as $(D^n,S^{n-1})$. This proves that $(P,P_{n-1})$ is a homology cell and concludes the induction step.
\end{proof}

The orientability conditions are subtle, if one allows finite stabilizers for an action. The next lemma explains this issue.

\begin{lem}\label{lemOrientability}
Let $X$ be a smooth closed connected orientable manifold and $T=T^k$ acts effectively on $X$. As before, let $Q=X/T$ be the orbit space, and $X_{k-1}$ be the union of all orbits of dimension $\leqslant k-1$, and $Q_{k-1}=X_{k-1}/T$. Then $Q\setminus Q_{k-1}$ is an orientable rational homology manifold.
\end{lem}

\begin{proof}
At first notice that the rational homology manifold $Y$ is orientable if and only if $H_{\dim Y}^c(Y;\Qo)$, top-degree homology with closed supports, is nonzero.

Orbits lying in the manifold $X\setminus X_{k-1}$ have dimension $k$, therefore the action is almost free on this space. For this reason we will denote $X\setminus X_{k-1}$ by $X^{\af}$ and $Q\setminus Q_{k-1}$ by $Q^{\af}$. Let $H\subset T$ denote the product of all stabilizers of the action on $X^{\af}$. The subgroup $H$ is finite since the action on a closed manifold $X$ has only finitely many stabilizers and the acting group is commutative. Consider two maps
\[
X^{\af}\stackrel{f_1}{\to} X^{\af}/H \stackrel{f_2}{\to} X^{\af}/T=Q^{\af}.
\]
Here $X^{\af}/H$ is the underlying space of an orbifold, and $f_2$ is a principal bundle with toric fiber $T/H$. The action of a finite abelian group $H$ preserves the orientation of $X$ since the action of $T$ does. By the slice theorem, each point of $X^{\af}/H$ has a neighborhood $U$ homeomorphic to the quotient of euclidean space by a finite abelian group action preserving the orientation, therefore $X^{\af}/H$ is a rational homology manifold. This quotient is orientable since the image of the fundamental class $[X^{\af}]\in H_{\dim X}^c(X;\Qo)$ under the proper map $X^{\af}\to X^{\af}/H$ is nonzero (its restriction to a neighborhood of any point is nonzero).

The additional quotient of $X^{\af}/H$ by a free $(T/H)$-action results in a homology manifold, since any point of $X^{\af}/H$ has a neighborhood of the form $U\times \Ro^k$, where $U$ is a neighborhood of its projection in $Q^{\af}$, and the local homology of $Q^{\af}$ can be easily computed. We claim that the base $Q^{\af}$ of this principal bundle with the toric fiber $T/H$ is orientable whenever its total space $X^{\af}/H$ is orientable.

Let us introduce more convenient notation to prove this fact. Assume a torus $T$ of dimension $d-s$ acts freely on an orientable rational homology manifold $E$ of dimension $d$, and $p\colon E\to B=E/T$ is the natural projection to the orbit space which is a rational homology manifold of dimension $s$. Let $x$ be an arbitrary point in $B$. The result of~\cite[Lm.17]{Yang} implies, for any sufficiently small neighborhood $U\subset B$ of $x$, that the tubular neighborhood $V=p^{-1}(U)\subset E$ of the orbit $Tx$ is homeomorphic to the direct product $U\times T$ in a way that $p$ is compatible with the projection to the first factor.

The orientation of $E$ implies the existence of a fundamental class $[E]\in H^c_d(E;\Qo)$ in the top degree homology with closed supports. This distinguished class determines distinguished element in the group $H^c_d(V\setminus Tx;\Qo)$ which is naturally isomorphic to $H_d(V,V\setminus Tx;\Qo)$. Since the pair $(V,V\setminus Tx)$ is naturally homeomorphic to the product $(U,U\setminus x)\times T$, we can consider the slant product
\[
/\colon H_d(U\times T,(U\setminus x)\times T;\Qo)\otimes H^{d-s}(T;\Qo)\to H_s(U,U\setminus x;\Qo).
\]
Taking the slant product $\alpha_x=[V\setminus Tx]/\omega$ with the fixed generator $\omega\in H^{d-s}(T;\Qo)\cong\Qo$ we obtain a distinguished non-zero element of the local homology group $H_s(U,U\setminus\{x\};\Qo)\cong H_s(B,B\setminus\{x\};\Qo)$. These local orientations are compatible, since the fundamental classes $[V\setminus Tx]$ are compatible for nearby orbits. Therefore, $B$ is an orientable rational homology manifold.
\end{proof}

The arguments similar to the proof of Lemma~\ref{lemMasPan} given above will be applied to prove Theorem~\ref{thmEqFormImpliesSphere}. We prove the following, more general statement.

\begin{prop}\label{propAcyclicity}
Assume that the action of $T^{n-1}$ on an orientable manifold $X=X^{2m}$, $m\geqslant n$ is equivariantly formal and has nonempty finite fixed point set. Assume that $\dim X_i=2i$ and $\dim Q_i=i$ for all $i<n-1$ (that is all proper face submanifolds carry actions of complexity zero). Then the following hold for the orbit space $Q=X^{2m}/T^{n-1}$:
\begin{enumerate}
  \item all proper faces $F$ of $Q$ are acyclic, i.e. $\Hr^*(F)=0$;
  \item the sponge $Q_{n-2}$ is $(n-3)$-acyclic, i.e. $\Hr^i(Q_{n-2})=0$ for $i\leqslant n-3$;
  \item the orbit space $Q$ is $n$-acyclic, i.e. $\Hr^i(Q)=0$ for $i\leqslant n$.
  \item $H^i(Q,Q_{n-2})=0$ for $i\leqslant n$ and $i\neq n-1$.
\end{enumerate}
Again, it is assumed that either $R=\Zo$ and all stabilizers of the action are connected, or $R=\Qo$.
\end{prop}

\begin{proof}
Since $X$ is equivariantly formal and $X^T$ is isolated, we have $H^{\odd}(X)=0$. Let $F$ be a proper $i$-dimensional face of $Q$, that is $i\leqslant n-2$. The face submanifold $X_F=p^{-1}(F)$ is a torus invariant submanifold of dimension $2i$ by assumption. According to \cite[Lm.2.2]{MasPan}, the manifold $X_F$ inherits the property of vanishing odd degree cohomology. Now, since $X_F$ is a manifold with an action of complexity zero, Lemma~\ref{lemMasPan} implies that its orbit space $F$ is acyclic. This proves item (1) of the proposition. Acyclicity of all proper faces also implies that the faces of $Q_{n-2}$ provide a structure of a homology cell complex on this space.

Now we write the ABFP-sequence \eqref{eqABseq} for $X$
\begin{multline}\label{eqABseqForX}
0\to H^*_T(X)\stackrel{i^*}{\to} H^*_T(X_0)\stackrel{\delta_0}{\to}
H^{*+1}_T(X_1,X_0)\stackrel{\delta_1}{\to}\cdots\\\cdots
\stackrel{\delta_{n-3}}{\to}H^{*+n-2}_T(X_{n-2},X_{n-3})\stackrel{\delta_{n-2}}{\to}H^{*+n-1}_T(X,X_{n-2})\to 0.
\end{multline}

Consider the $i$-th term in \eqref{eqABseqForX} with $i\leqslant n-2$:
\begin{equation}\label{eqTech1}
H^{*+i}_T(X_{i},X_{i-1})\cong \bigoplus_{F\colon\dim F=i}H^{*+i}_T(X_F,X_F\cap X_{i-1})\cong \bigoplus_{F\colon\dim F=i}H^{i}(F,\dd F)\otimes H^*(BT_F).
\end{equation}
The latter isomorphism is due to the following facts: (1) the action of $T^{n-1}/T_F$ on $X_F\setminus X_{i-1}$ is (almost) free, (2) $(F,\dd F)$ is a homology cell. Similarly, for the rightmost term in \eqref{eqABseqForX} we have
\begin{equation}\label{eqTech2}
H^{*+n-1}_T(X,X_{n-2})\cong H^{*+n-1}(Q,Q_{n-2}) 
\end{equation}
since the torus action is (almost) free on $Q\setminus Q_{n-2}$. Specializing~\eqref{eqABseqForX} to a degree $*<0$, we get the exact sequence
\[
0\to 0\to\cdots \to 0\to H^{*+n-1}_T(X,X_{n-2})\to 0.
\]
This observation shows that
\begin{equation}\label{eqTechAcycLow}
H^i(Q,Q_{n-2})=0 \mbox{ for } i\leqslant n-2.
\end{equation}
From \eqref{eqTech1} and \eqref{eqTech2}, it follows that the ABFP-sequence can be written in the form
\begin{multline}\label{eqABseqInFaces}
0\to H^*_T(X)\stackrel{i^*}{\to} \bigoplus_{F\colon \dim F=0}H^*(BT_F)\stackrel{\delta_0}{\to}\cdots\\\cdots
\stackrel{\delta_{n-3}}{\to}\bigoplus_{F\colon\dim F=n-2}H^*(BT_F)\stackrel{\delta_{n-2}}{\to}H^{*+n-1}(Q,Q_{n-2})\to 0.
\end{multline}
Specializing to degree 0 (the lowest nontrivial degree) in each module, we get
\begin{equation}\label{eqABseqAtDeg0}
0\to R\to \bigoplus_{F\colon \dim F=0}R\stackrel{\delta_0}{\to}\cdots
\stackrel{\delta_{n-3}}{\to}\bigoplus_{F\colon \dim F=n-2}R\stackrel{\delta_{n-2}}{\to}H^{n-1}(Q,Q_{n-2})\to 0,
\end{equation}
where $\delta_i$ for $i\leqslant n-2$ is the connecting homomorphism in the cohomological exact sequence of the triple $(Q_{i+1},Q_i,Q_{i-1})$. The truncated sequence
\begin{equation}\label{eqABseTrunc}
0\to R\to \bigoplus_{F\colon\dim F=0}R\stackrel{\delta_0}{\to}\cdots
\stackrel{\delta_{n-3}}{\to}\bigoplus_{F\colon\dim F=n-2}R\to 0,
\end{equation}
is the reduced complex of cellular cochains of the homology cell complex $Q_{n-2}$:
\[
\bigoplus_{F\colon\dim F=i}R\cong H^i(Q_i,Q_{i-1}).
\]
Hence, acyclicity of the ABFP-sequence implies that
\begin{equation}\label{eqTechAcycLow2}
\Hr^i(Q_{n-2})=0 \mbox{ for } i\leqslant n-3.
\end{equation}
This proves item (2) of the proposition.

The acyclicity of \eqref{eqABseqAtDeg0} at the last two terms implies that the induced homomorphism
\[
\widetilde{\delta}_{n-2}\colon \left(\bigoplus\nolimits_{F\colon\dim F=n-2}R\right)/\Imm\delta_{n-3}\to H^{n-1}(Q,Q_{n-2})
\]
is an isomorphism. However, according to \eqref{eqABseTrunc} the module $\left(\bigoplus_{F\colon\dim F=n-2}R\right)/\Imm\delta_{n-3}$ coincides with the cellular cohomology module $H^{n-2}(Q_{n-2})$ and the map $\widetilde{\delta}_{n-2}$ is induced by $\delta_{n-2}$ by passing to cellular cohomology. Therefore,
\begin{equation}\label{eqTechAcycInterest}
\widetilde{\delta}_{n-2}\colon H^{n-2}(Q_{n-2})\to H^{n-1}(Q,Q_{n-2})
\end{equation}
is an isomorphism. It is easy to check that $\widetilde{\delta}_{n-2}$ coincides with the connecting homomorphism in the long exact sequence of the pair $(Q,Q_{n-2})$.

Putting $*=1$ in \eqref{eqABseqInFaces}, we have
\begin{equation}\label{eqTechAcycHigh}
H^{n}(Q,Q_{n-2})=0,
\end{equation}
since the previous term $\bigoplus_{F\colon\dim F=n-2}H^*(BT_F)$ vanishes for odd $*$. Item (4) of the proposition is justified by~\eqref{eqTechAcycLow} and~\eqref{eqTechAcycHigh}.

Gathering \eqref{eqTechAcycLow}, \eqref{eqTechAcycLow2}, \eqref{eqTechAcycInterest}, \eqref{eqTechAcycHigh} together, we see that the connecting homomorphisms
\[
\Hr^{i-1}(Q_{n-2})\to H^i(Q,Q_{n-2})
\]
in the long exact sequence of the pair $(Q,Q_{n-2})$ are isomorphisms for all $i\leqslant n$. Hence $\Hr^i(Q)=0$ for $i\leqslant n$. This proves item (3) of the proposition.

The $\Qo$-version of the proposition follows the same lines, since we only used ABFP-sequence, which is exact over $\Qo$ if disconnected stabilizers are allowed.
\end{proof}

As a corollary, we obtain a proof of Theorem~\ref{thmEqFormImpliesSphere}.

\begin{proof}[Proof of Theorem~\ref{thmEqFormImpliesSphere}]
If a torus $T^{n-1}$ acts on an orientable manifold $X^{2n}$ in general position, then, according to \cite{AyzCompl} all proper face submanifolds carry an action of complexity zero (also see Lemma~\ref{lemJgenForFace} below). Therefore, Proposition~\ref{propAcyclicity} applies. Now, the orbit space $Q$ is a closed topological manifold of dimension $(n+1)$ by~\cite[Thm.2.10]{AyzCompl}, while the subspace $Q_{n-2}$ has dimension $n-2$. We have $H_{n+1}^c(Q)\cong H_{n+1}^c(Q\setminus Q_{n-2})$, and the latter group is nonzero since $Q^{\af}=Q\setminus Q_{n-2}$ is orientable by Lemma~\ref{lemOrientability}. The $(n+1)$-dimensional closed orientable manifold $Q$ is $n$-acyclic by Proposition~\ref{propAcyclicity}. Hence $Q$ is a homology sphere by Poincar\'{e}--Lefschetz duality.
\end{proof}

\section{General actions in $j$-general position}\label{secNonGeneral}

In this section we prove Theorem~\ref{thmOrbitAcyclicityAnyComplexity}. Assume that a torus $T=T^k$ acts effectively on a connected closed smooth manifold $X=X^{2n}$ and $H^{\odd}(X)=0$. The action is equivariantly formal and has complexity $n-k$. At first, we give some comments on actions in $j$-general position. Note that the action is in $1$-general position if and only if all its weights are nonzero. This means that fixed points of the action are isolated. It will be assumed that $j\geqslant 1$, so that all actions under consideration have finite sets of fixed points.

Theorem~\ref{thmOrbitAcyclicityAnyComplexity} will be proved by induction on $k$. Let $F$ be a face of $Q=X/T$. The number $\dim (T/T_F)$ will be called the rank of $F$ and denoted $\rk F$. Therefore the filtration term $Q_i$ is the union of all faces of rank $i$. The complexity of the action on $X_F$ will be denoted
\[
\com F=\frac{1}{2}\dim X_F-\dim(T/T_F)=\frac{1}{2}(\dim F-\rk F).
\]
To perform the induction argument we need a technical but simple statement.

\begin{lem}\label{lemJgenForFace}
Assume that an action of $T=T^k$ on $X=X^{2n}$ is in $j$-general position, $j\leqslant k$. Let $F$ be a face of $Q=X/T$ and $X_F\subset X$ be the corresponding face submanifold. Then the following hold:
\begin{enumerate}
  \item $\com(F)\leqslant\com(Q)=n-k$;
  \item for every face $F$ of rank $<j$, the action of $T/T_F$ on $X_F$ has complexity zero;
  \item for every face $F$ of rank $j$, the action of $T/T_F$ on $X_F$ is in $j$-general position.
\end{enumerate}
\end{lem}

\begin{proof}
Let $\ttt$ and $\ttt_F$ be the Lie algebras of $T$ and $T_F$ respectively, so that $\ttt\cong \Ro^k$ and $\ttt_F\cong \Ro^{k-\rk F}$. Let $\alpha_1,\ldots,\alpha_n\in\Hom(T,T^1)$ be the tangent weights of the action at some fixed point $x\in X_F\subset X$. The weights of the induced action of $T$ on $X_F$ are given by some subset $\{\alpha_i\}_{i\in A}$, $A\subset[n]$. Since $T_F$ fixes $X_F$ pointwise, the identity $\langle w, \alpha_i\rangle=0$ holds for any $w\in \ttt_F$ and any $i\in A$. Here we assume that the weight lattice $\Hom(T,T^1)$ is naturally embedded in $\ttt^*$. Therefore, the vectors $\{\alpha_i\}_{i\in A}$ lie in the annihilator $\ttt_F^{\bot}\cong \Ro^{\rk F}$.

The vectors $\{\alpha_i\}_{i\in [n]}$ linearly span the space $\ttt^*$ since the action of $T$ on $X$ is effective (if  $\{\alpha_i\}_{i\in [n]}$ do not span $\ttt^*$, the nonzero subspace $\bigcap_{i\in[n]}\Ker \alpha_i$ would be the tangent Lie algebra of the noneffective kernel of the action). Similarly, the vectors $\{\alpha_i\}_{i\in A}$ linearly span the space $\ttt_F^{\bot}$ since the action of $T/T_F$ on $X_F$ is effective. Therefore the complement $[n]\setminus A$ contains at least $\dim\ttt^*-\dim\ttt_F^{\bot} = k-\rk F$ elements. Hence
\[
\com Q -\com F=(n-k)-(|A|-\rk F)\geqslant 0
\]
which proves item (1). Now let $\rk F<j$. This condition together with $j$-generality implies that any $\rk F+1\leqslant j$ weights are linearly independent. If $|A|=\frac{1}{2}\dim X_F>\rk F$, then $X_F$ has at least $\rk F+1$ many weights at a fixed point. These weights lie in the space $\ttt_F^{\bot}$ of dimension $\rk F$, hence they are linearly dependent which gives a contradiction. Hence $\frac{1}{2}\dim X_F=\rk F$ and therefore $X_F$ has complexity 0 which proves~(2). If $\rk F\geqslant j$, then every $j\leqslant \rk F$ of the weights $\{\alpha_i\}_{i\in A}$ are linearly independent, therefore the induced effective action of $T/T_F$ on $X_F$ is in $j$-general position by definition, which proves~(3).
\end{proof}

This lemma implies

\begin{lem}\label{lemJgeneralTech}
Under the assumptions of Theorem~\ref{thmOrbitAcyclicityAnyComplexity}, the following acyclicity conditions hold:
\begin{enumerate}
  \item $H^s(Q_i,Q_{i-1})=0$ for $i<j$ and $s\neq i$;
  \item $H^s(Q_j,Q_{j-1})=0$ for $s<j$ and for $s=j+1$.
\end{enumerate}
\end{lem}

\begin{proof}
According to Lemma~\ref{lemJgenForFace}, all faces of the $(j-1)$-skeleton $Q_{j-1}$ correspond to complexity zero case. Therefore $Q_{j-1}$ is a homology cell complex by Lemma~\ref{lemMasPan}. This proves (1).

Further, we have $H^s(Q_j,Q_{j-1})=\bigoplus_{\rk F=j}H^s(F,F_{-1})$. As before, $F_{-1}$ and $(X_F)_{-1}$ denote the union of faces, resp. face submanifolds of lower rank, see Construction~\ref{conLowerRank}. According to Lemma~\ref{lemJgenForFace}, each proper face submanifold of $X_F$ carries the action of complexity 0. Therefore, we can apply Proposition~\ref{propAcyclicity} for each face manifold $X=X_F$ having rank $n-1=j$. Item (4) of Proposition~\ref{propAcyclicity} shows that
\[
H^s(Q_j,Q_{j-1})=\bigoplus_{\rk F=j}H^s(F,F_{-1})=0
\]
for $s<j$ and for $s=j+1$.
\end{proof}

Now we prove one more statement concerning acyclicity of certain relative pairs. The next lemma generalizes one of the arguments used in the proof of Theorem~\ref{thmEqFormImpliesSphere} to actions of arbitrary complexity.

\begin{lem}\label{lemRelativeVanish}
The relative cohomology modules $H^i(F,F_{-1})$ vanish for $i<\rk F$.
\end{lem}

\begin{proof}
The proof goes by induction on $\rk F$. If $\rk F=0$ then $F$ is nonempty and $F_{-1}=\varnothing$, so for $i<0$ there is nothing to prove. Now assume that the statement holds for all $F$ with $\rk F<s$ and prove it for $F=Q$, $\rk Q=s$. As in Section~\ref{secProofMain}, we write down the ABFP-sequence for the manifold $X$ over $Q$:
\begin{equation}\label{eqABseqXF}
0\to H^*_T(X)\stackrel{i^*}{\to} H^*_T(X_0)\stackrel{\delta_0}{\to}\cdots
\stackrel{\delta_{s-2}}{\to}H^{*+s-1}_T(X_{s-1},X_{s-2})\stackrel{\delta_{s-1}}{\to}H^{*+s}_T(X,X_{s-1})\to 0
\end{equation}
Since $X$ is equivariantly formal, the sequence is exact. Further, we have
\[
H^*_T(X_{s-1},X_{s-2})\cong \bigoplus_{F\colon \rk F=s-1} H^*_T((X_F),(X_F)_{-1})\cong \bigoplus_{F\colon \rk F=s-1} H^*(F,F_{-1})\otimes H^*(BT_F).
\]
The group $H^i(F,F_{-1})$ vanishes for $i<s-1$ by induction hypothesis. Hence $H^i_T(X_{s-1},X_{s-2})$ vanishes for $i<s-1$ as well. Specializing \eqref{eqABseqXF} to $*<0$, we deduce that $H^i_T(X,X_{s-1})\cong H^i(Q,Q_{-1})$ vanishes for $i<s=\rk Q$.
\end{proof}

\begin{proof}[Proof of Theorem~\ref{thmOrbitAcyclicityAnyComplexity}]
Consider the cohomological spectral sequence associated with the filtration $\{Q_i\}$ of $Q$:
\[
E^{p,q}_1\cong H^{p+q}(Q_p,Q_{p-1}) \Rightarrow H^{p+q}(Q).
\]
The terms $E^{p,q}_1$ with $q<0$ vanish by Lemma \ref{lemRelativeVanish}. The terms $E^{p,q}_1$ with $p<j$ and $q>0$ vanish by item (1) of Lemma~\ref{lemJgeneralTech}. The term $E_1^{j,1}$ vanishes by item (2) of Lemma \ref{lemJgeneralTech}. The complex $(E^{p,0}_1,d_1)$ is nontrivial. However this differential complex coincides with the ABFP-sequence \eqref{eqABseqXF} specialized at degree $0$. Since ABFP-sequence is exact, we have $E_2^{p,0}=0$. These considerations show that $E_2^{p,q}=0$ for $p+q\leqslant j+1$, which proves the theorem.
\end{proof}

So far, in general, there is a topological restriction on the orbit spaces of equivariantly formal actions with isolated fixed points: they are always $2$-acyclic. If the action is in $j$-general position, then the orbit space is $(j+1)$-acyclic. From the homological point of view, however, this is the only restriction which we can obtain, at least for the actions of complexity one. In the joint work \cite{AyzCher} of the first author and Cherepanov, the following statement was proved.

\begin{prop}[{\cite[Thm.2]{AyzCher}}]\label{propAyzCherep}
For any finite simplicial complex $L$, there exists a closed smooth manifold $X^{2n}$ with $H^{\odd}(X^{2n})=0$, and the action of $T^{n-1}$ in $j$-general position, $j\geqslant 1$, such that the orbit space $Q^{n+1}=X^{2n}/T^{n-1}$ is homotopy equivalent to the $(j+2)$-fold suspension $\Sigma^{j+2}L$.
\end{prop}

An example can be constructed as a certain $\CP^1$-bundle over the permutohedral variety. The torus action on this bundle is induced by the torus action on the permutohedral variety in the base. This manifold is a smooth projective toric variety, hence the action is Hamiltonian and cohomologically equivariantly formal.

\section{A criterion of equivariant formality in complexity one: case $n=2$}\label{secSphereImpliesFormalitySimple}

Our next goal is to formulate and prove the theorem converse to Theorem~\ref{thmEqFormImpliesSphere} that is the criterion for equivariant formality of torus actions of complexity one in general position in terms of the orbit space structure. In this section we discuss the case $n=2$, i.e. the $T^1$-action on 4-dimensional manifolds. This case is simpler but reflects some of the main ideas of the general case. The general theorem is stated and proved in Section~\ref{secSphereImpliesFormality}.

\begin{thm}\label{thmSphereImpliesEqFormSimple}
Let the coefficient ring $R$ be either $\Zo$ or $\Qo$. Assume that an effective smooth action of $T^1$ on a closed orientable manifold $X=X^4$ satisfies the following properties:
\begin{enumerate}
  \item the action has nonempty finite set $X_0$ of fixed points;
  \item the action is semifree (that is, the action is free on the complement $X\setminus X_0$);
  \item the orbit space $Q=Q^3=X^4/T^1$ is a homology $3$-sphere: $\Hr_i(Q)=0$ for $i=0,1,2$, and $H_3(Q)\cong R$.
\end{enumerate}
Then the action is equivariantly formal: $H^{\odd}(X)=0$.
\end{thm}

It should be noted that circle actions on 4-folds are a classical subject in algebraic topology~\cite{ChL,Fint}, in particular, the relation between simple connectedness of $X^4$ and simple connectedness of the orbit 3-fold $Q^3=X^4/T^1$ as well as the classification of $T^1$-manifolds of dimension $4$ in terms of their orbit spaces was studied in detail by Fintushel~\cite{Fint} (also see references therein). Theorem~\ref{thmSphereImpliesEqFormSimple} is a homological version of his result. We suppose that the reasoning below is not the easiest way to prove the statement, however it demonstrates the key ideas to be used in the proof of Theorem~\ref{thmSphereImpliesEqForm} below, which tackles the case of general~$n$.


\begin{proof}[Proof of Theorem~\ref{thmSphereImpliesEqFormSimple}]
By assumption, we have a $T^1$-action on a $4$-manifold $X^4$, the orbit space $Q^3$ is a homology 3-sphere and there is a nonempty finite set $Z=Q_0$ of fixed points. The truncated ABFP-sequence has the form
\begin{equation}\label{eqN2truncABFP}
0\to H^*_{T^1}(Z)\to H^{*+1}_{T^1}(X^4,Z)\to 0,
\end{equation}
or, equivalently,
\begin{equation}\label{eqABcase4}
0\to \bigoplus_{x\in Z}H^*(BT^1)\stackrel{\delta_0}{\to} H^{*+1}(Q,Z)\to 0.
\end{equation}
Since $Q$ is a homology 3-sphere, and $Z$ is a finite set, the group $H^k(Q,Z)$ is nontrivial only for $k=1$ and $3$. The isomorphisms $H^3(Q,Z)\cong\Zo$, $H^1(Q,Z)\cong \Hr^0(Z)$ follow from the long exact sequence of the pair $(Q,Z)$. The homomorphism $\delta_0\colon \bigoplus_{x\in Z}H^0(BT^1)\to H^1(Q,Z)$ can be identified with the natural homomorphism $H^0(Z)\to \Hr^0(Z)$ which is surjective.

The homomorphism $\delta_0\colon \bigoplus_{x\in Z}H^2(BT^1)\to H^3(Q,Z)\cong \Zo$ is surjective as well. Indeed, let $x\in Z$ be a fixed point, $U_x\subset Q$ be a small disk neighborhood of $x$ in $Q$, and $W_x$ be the preimage of $U_x$ in $X^4$, this is a small disk neighborhood of $x$ in $X^4$. By collapsing the subset $X\setminus W_x$, we get a 4-sphere $S^4_x=X/(X\setminus W_x)$ with the induced action of $T^1$.
The orbit space $S^4_x/T^1$ is homeomorphic to $Q/(Q\setminus U_x)\cong S^3$. The $T^1$-action on $S^4_x$ has a fixed point set $\tilde{Z}$ consisting of two points, $\tilde{Z}=\{x,\infty\}$. Since the $T^1$-action on $S^4_x$ is equivariantly formal, the ABFP-sequence for $S^4_x$ is exact. Hence the map $H^2_{T^1}(\tilde{Z})\to H^3_{T^1}(S^4_x,\tilde{Z})$ is surjective. We have the following commutative diagram
\[
\xymatrix{
H^2_{T^1}(\tilde{Z},\infty)\ar@{->>}[r] \ar@{^{(}->}[d] & H^3_{T^1}(S^4_x,\tilde{Z})\ar@{->}[d]^{\cong}\\
H^2_{T^1}(Z)\ar@{->}[r] & H^3_{T^1}(X,Z)
}
\]
where the vertical arrows are induced by collapsing $X\setminus W_x$ and passing to cohomology relative to the point $\infty$ (which is the class of the collapsed subset $X\setminus W_x$). The right vertical arrow is an isomorphism since
\[
H^3_{T^1}(X,Z)\cong H^3(Q,Z)\cong \Zo\cong H^3(S^3,\tilde{Z})\cong H^3_{T^1}(S^4_x,\tilde{Z}).
\]
The left vertical arrow is the inclusion of the summand $H^2_{T^1}(\tilde{Z},\infty)\cong H^2(BT^1)$ corresponding to $x\in Z$ into $H^2_{T^1}(Z)\cong \bigoplus_{y\in Z}H^2(BT^1)$. It follows that the lower horizontal map is an epimorphism.

The long exact sequence of equivariant cohomology of the pair $(X,Z)$ splits into short exact sequences since all homomorphisms $\delta_0\colon H^*_T(Z)\to H^{*+1}_T(X,Z)$ are surjective. Therefore the following ABFP-sequence is exact
\[
0\to H^*_{T^1}(X^4)\to H^*_{T^1}(Z)\to H^{*+1}_{T^1}(X^4,Z)\to 0.
\]
By the result of Franz--Puppe~\cite[Thm.1.1]{FP}, this condition implies that $X^4$ is equivariantly formal.
\end{proof}

\begin{rem}\label{remCompSupp}
An open disk neighborhood $W_x$ of a fixed point $x$ is chosen to localize the study of cohomology in the vicinity of $x$. If we denote the closure of $W_x$ by $\overline{W}_x$ and its boundary by $\dd W_x$, the sphere $S_x^4$ in the arguments above becomes the quotient $\overline{W}_x/\dd W_x$, the one-point compactification of $W_x$. The relative cohomology module $H^*(\overline{W}_x,\dd W_x)\cong \Hr^*_T(S_x)$ is naturally isomorphic to the equivariant cohomology with compact supports $H^*_{T,c}(W_x)$. Similarly, the relative cohomology $H^*_T(S_x^4,\tilde{Z})$ can be replaced with its compactly supported version $H^*_{T,c}(W_x,W_x\cap Z)$ of the neighborhood $W_x$ itself.

In the proof of Theorem~\ref{thmSphereImpliesEqFormSimple} and in the arguments to follow, there is no actual need to take one-point compactifications of the neighborhoods of points. Similar arguments work fairly well for cohomology with compact supports. All results concerning acyclicity of ABFP sequence are valid for cohomology with compact supports, according to~\cite[Sec.4.1]{AFP}. However, we prefer to work with spheres, the compactifications of neighborhoods, for the reason that cohomology with compact support seems less geometrically intuitive to us than relative cohomology of finite CW-pairs.
\end{rem}

\section{Topology of sponges}\label{secSpongesHomology}

In order to prove the analogue of Theorem~\ref{thmSphereImpliesEqFormSimple} for actions of $T^{n-1}$ on $X^{2n}$ in general position, we need a deeper insight into the structure of orbit type filtrations of such actions. The general theory of sponges was developed in \cite{AyzCompl}. In this section we recall the basic definitions and examples, and prove a collection of technical homological lemmas.

\begin{con}\label{conModelSponge}
Let $v_1,\ldots,v_{n-1}$ be a basis of the vector space $\Ro^{n-1}$ and $v_n=-\sum_{i=1}^{n-1}v_i$. Consider the subset $C^{n-2}$ of
$\Ro^{n-1}$ given by
\[
C^{n-2}=\bigcup_{I\subset [n],|I|=n-2} \Cone(v_i\mid i\in I).
\]
The subset $C^{n-2}$ is the $(n-2)$-skeleton of the real simplicial fan corresponding to the toric variety $\CP^{n-1}$. Schematic figures of $C^{n-2}$ can be found in~\cite{AyzCompl}. The subset $C^{n-2}$ comes equipped with the filtration $C_0\subset\cdots\subset C_{n-2}=C^{n-2}$, where $C_k$ is the union of $k$-dimensional cones of the fan $C^{n-2}$. A point $x\in C^{n-2}\subset \Ro^{n-1}$ is said \emph{to have type $k$} if $C^{n-2}$ cuts a small disc $U_x\subset\Ro^{n-1}$ around $x$ into $n-k$ disjoint chambers. The filtration term $C_k$ consists of all points of type $\leqslant k$.
\end{con}
%

\begin{lem}\label{lemLocHomOfC}
Let $x\in C^{n-2}$ be a point of type $k$. Then the local cohomology group $H^j(C^{n-2},C^{n-2}\setminus \{x\})$ vanishes for $j\neq n-2$ and $H^{n-2}(C^{n-2},C^{n-2}\setminus \{x\})\cong \Zo^{n-1-k}$.
\end{lem}

\begin{proof}
If $x$ has type $0$, that is $x$ is the origin of $\Ro^{n-1}$, then $H^*(C^{n-2},C^{n-2}\setminus \{x\})\cong H^*(\Cone \Delta_{n-1}^{(n-3)},\Delta_{n-1}^{(n-3)})\cong \Hr^{*-1}(\Delta_{n-1}^{(n-3)})$, where $\Delta_{n-1}^{(n-3)}$ is the $(n-3)$-skeleton of an $(n-1)$-dimensional simplex. In this case, the computation of cohomology is a simple exercise. In general, if $x$ has type $k$, then $x$ has a neighborhood homeomorphic to $\Ro^k\times C^{n-1-k}$ and the statement follows from the type $0$ case and the suspension isomorphism.
\end{proof}

Let us recall a notion of the sponge, introduced in~\cite{AyzCompl}. This notion models the structure of orbit type filtration for torus actions of complexity one in general position.

\begin{defin}
Let $Q=Q^{n+1}$ be a closed topological manifold and $Z\subset Q$ its subspace. A pair $(Q^{n+1},Z^{n-2})$ is called \emph{a sponge} if, for any point $x\in Z$, there is a neighborhood $U_x\subset Q^{n+1}$ such that $(U_x,U_x\cap Z^{n-2})$ is homeomorphic to $(V\times \Ro^2,V\cap C^{n-2})$, where $V$ is an open subset of the space $\Ro^{n-1}$ and $C^{n-2}$ is the model space defined in Construction~\ref{conModelSponge}.
\end{defin}

Sometimes the space $Z^{n-2}$ itself will be called a sponge. The filtration $\{C_k\}$ on $C^{n-2}$ naturally induces the filtration $Z_0\subset Z_1\subset\cdots\subset Z_{n-2}$ of $Z^{n-2}$. A point $x\in Z$ is said \emph{to have type $k$} if it lies in $Z_k\setminus Z_{k-1}$. The type of a point is well-defined by Lemma~\ref{lemLocHomOfC}, therefore the whole filtration $\{Z_k\}$ is well-defined. The closures of connected components of $Z_k\setminus Z_{k-1}$ are called \emph{(proper) $k$-faces} of the sponge. We recall from \cite[Prop.2.16]{AyzCompl}, that whenever an action of $T^{n-1}$ on $X^{2n}$ is in general position and satisfies~\eqref{eqCondJoint}, the pair $(Q,Q_{n-2})$ is a sponge. Here, as before, $Q=X^{2n}/T^{n-1}$ is the orbit space, and $Q_{n-2}$ is its orbit $(n-2)$-skeleton. The notion of faces for the orbit type filtration and that for a sponge are consistent.

\begin{ex}\label{exExamplesSponges}
We have the following natural examples of sponges.
\begin{enumerate}
  \item Assume there is a locally standard action of $T^n$ on a manifold $M^{2n}$, so the orbit space $P=P^n=M^{2n}/T^n$ is a manifold with corners. Assume that the induced action of a subtorus $T^{n-1}\subset T^n$ on $M^{2n}$ has isolated fixed points, and it is in general position. Then the sponge of $T^{n-1}$-action on $M^{2n}$ is an $(n-2)$-skeleton of $P$, see~\cite{AyzCompl}. In particular, it was proved in~\cite{AyzCompl}, that whenever $M^{2n}$ is a quasitoric manifold, its orbit space $M^{2n}/T^{n-1}$ is homeomorphic to the sphere $S^{n+1}$, and the sponge is the $(n-2)$-skeleton of the orbit polytope (that is the boundary of a simple polytope minus the interiors of all facets).
  \item The following actions were mentioned in the introduction: the $T^3$-action on the Grassmann manifold $G_{4,2}$ of complex $2$-planes in $\Co^4$, the $T^2$-action on the manifold $F_3$ of full complex flags in $\Co^3$, and the $T^3$-action on the quaternionic projective plane $\HP^2$. Their sponges are shown on Fig.~\ref{figSpongeExamples}. The sponges for $G_{4,2}$ and $F_3$ were described in~\cite{AyzCompl}, while the sponge for $\HP^2$ was described in detail in~\cite{AyzHP}.
\end{enumerate}
\end{ex}

\begin{figure}[h]
\begin{center}
\includegraphics[scale=0.22]{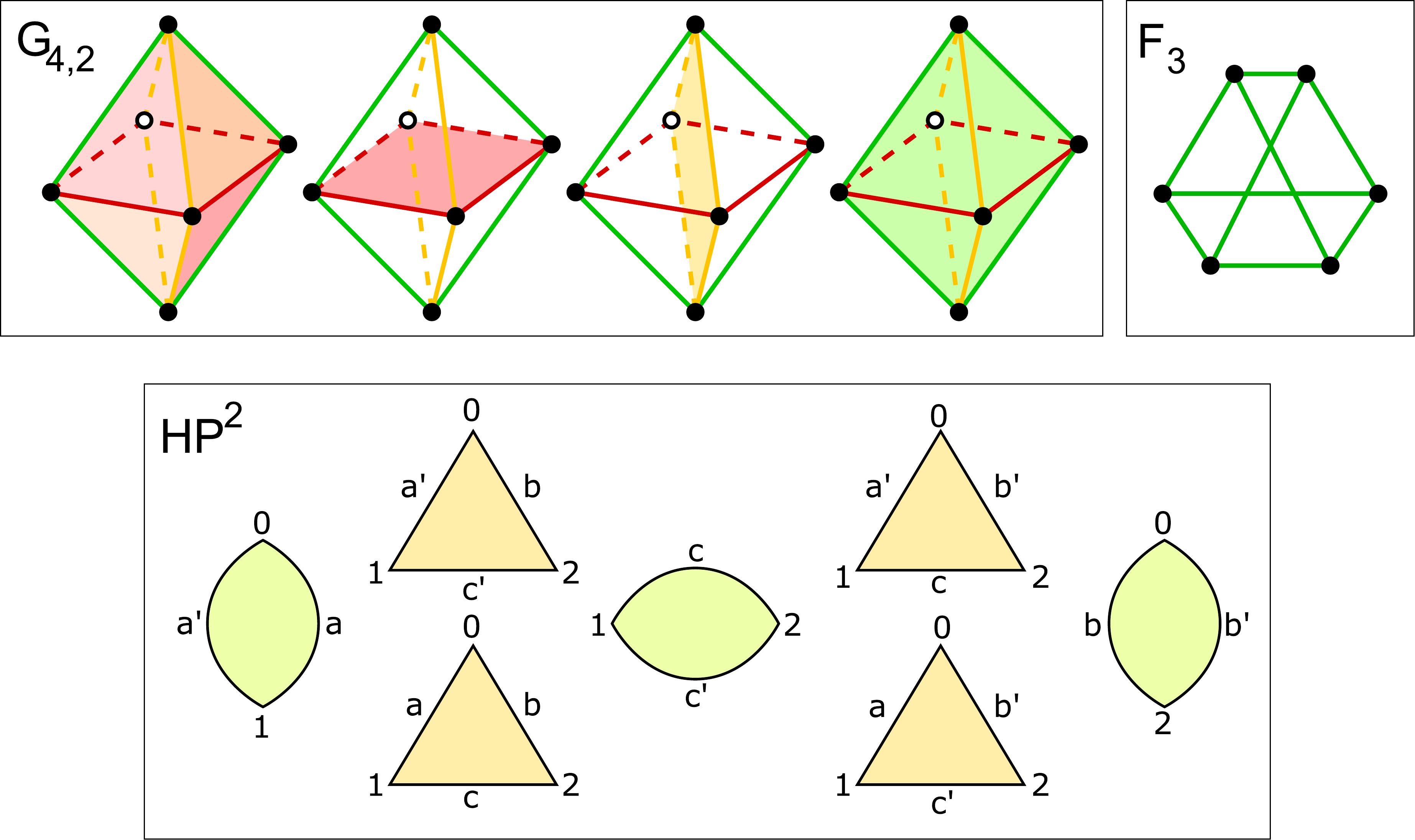}
\end{center}
\caption{The sponges of complexity one actions on $G_{4,2}$, $F_3$, and $\HP^2$. The sponge of $G_{4,2}$ is the boundary of octahedron with 3 additional square faces attached along equatorial circles. The sponge of $F_3$ is the complete bipartite graph $K_{3,3}$ (this is the GKM-graph of $F_3$). The sponge of $\HP^2$ has 7 2-dimensional faces glued together according to the labels. These figures appeared in the works~\cite{AyzCompl,AyzHP} of the first author.}\label{figSpongeExamples}
\end{figure}

In view of Theorem~\ref{thmEqFormImpliesSphere}, it is natural to introduce the following definition.

\begin{defin}\label{definAcyclicSponge}
A sponge $(Q^{n+1},Z^{n-2})$ is called \emph{acyclic} if the following conditions hold: (1) $Q^{n+1}$ is a homology $(n+1)$-sphere; (2) each face of $Z^{n-2}$ is acyclic; (3) the space $Z^{n-2}$ is $(n-3)$-acyclic.
\end{defin}

In this section we work with coefficients in $\Zo$. However, Definition~\ref{definAcyclicSponge} makes sense for coefficients in any field $\ko$ as well.

\begin{rem}
If we are given just the space $Z^{n-2}$ without specifying the ambient manifold $Q^{n+1}$, then we call a sponge $Z$ acyclic, if conditions (2) and (3) above hold.
\end{rem}

For a sponge $Z^{n-2}$, consider the poset $S_Z$ of proper faces of $Z^{n-2}$, ordered by inclusion. In general, if $S$ is a poset, we use the notation
\[
S_{\leqslant s}=\{t\in S\mid t\leqslant s\},\qquad S_{\geqslant s}=\{t\in S\mid t\geqslant s\},
\]
for $s\in S$. The posets $S_{<s}$, $S_{>s}$, $S_{(s_1,s_2)}$, etc. are defined similarly. For example $S_{(s_1,s_2)}=\{t\in S\mid s_1<t<s_2\}$.

\begin{rem}\label{remDualSimplicial}
A poset $S$ is called \emph{dually simplicial}, if it has the unique greatest element, and, for any $s\in S$, the subposet $S_{\geqslant s}$ is isomorphic to the boolean lattice. For an action of complexity 0 of $T=T^n$ on $M=M^{2n}$, the orbit space $P$ is a nice manifold with corners (in the terminology of~\cite{MasPan}). The poset $S_P$ of faces of $P$ (which coincides with the poset of face submanifolds of $M$) is dually simplicial. This follows from the fact that the poset of faces of the nonnegative cone $\Rg^n$ is isomorphic to the boolean lattice.
\end{rem}

To formulate several next results we need to recall the notions from the combinatorial topology and algebraic combinatorics of posets.

\begin{con}
A simplicial complex $K$ on a (finite) vertex set $V$ is a collection of subsets of $V$, such that $\varnothing\in K$ and $I\in K$, $J\subset I$ implies $J\in K$. Let $|K|$ denote the geometrical realization of $K$, this is a finite CW-complex corresponding to $K$. One can speak about topological characteristics of $K$ via the geometrical realization. For example, $K$ is called acyclic, if the space $|K|$ is acyclic. If $I\in K$ is a simplex, the simplicial complex $\link_KI=\{J\subset V\mid J\cap I=\varnothing, J\cup I\in K\}$ is called \emph{the link} of $I$ in $K$. In particular, $\link_K\varnothing=K$. A simplicial complex $K$ (of dimension $d$) is called \emph{Cohen--Macaulay}, if the following conditions hold: (1) $K$ is $(d-1)$-acyclic, (2) for any simplex $I\in K$, the complex $\link_KI$ is $(d-1-|I|)$-acyclic. The definition depends on the ground ring of coefficients.

Let $S$ be a finite poset. Consider the simplicial complex $\ord(S)$ called \emph{the order complex of $S$} is defined as follows. The vertex set of $\ord(S)$ is $S$. Simplices of $\ord(S)$ are the subsets of pairwise comparable elements of $S$. In other words, each chain $s_0<s_1<\cdots<s_k$ in $S$ is a $k$-dimensional simplex $\sigma=\{s_0,\ldots,s_k\}$ in $\ord(S)$. Properties of simplicial complexes can be transferred to finite posets via the construction of the order complex. In particular, the geometrical realization $|S|$ of a poset $S$ is the geometrical realization of its order complex. A poset $S$ is called \emph{Cohen--Macaulay}, if $\ord(S)$ is Cohen--Macaulay.
\end{con}

\begin{lem}\label{lemSZCohenMacaulay}
If $Z^{n-2}$ is an acyclic sponge, then the geometrical realization $|S_Z|$ is $(n-3)$-acyclic. Moreover, $S_Z$ is a Cohen--Macaulay poset.
\end{lem}

\begin{proof}
The proof essentially repeats the idea of \cite[Prop.5.14]{MasPan}. 
We consider the filtration $\{|S_Z|_k\}$ of $|S_Z|$, where $|S_Z|_k$ is the geometrical realization of the subposet $\{F\in S_Z\mid \dim F\leqslant k\}$. There exists a map $f\colon Z^{n-2}\to |S_Z|$ preserving the filtrations on these spaces. Indeed, the map can be constructed inductively: at each step, we need to extend the given map $f\colon \dd F\to |(S_Z)_{<F}|$ to the map from $F$ to $|(S_Z)_{\leqslant F}|$. Such extension exists since $(F,F_{-1})$ is a CW-pair, and the target space $|(S_Z)_{\leqslant F}|=\Cone |(S_Z)_{<F}|$ is contractible. It is natural to call the subsets $|(S_Z)_{\leqslant F}|$ the faces of $|S_Z|$.

Since both spaces $Z^{n-2}$ and $S_Z$ have acyclic faces, the constructed map $f\colon Z^{n-2}\to |S_Z|$ induces the isomorphism of the (co)homology spectral sequences corresponding to the filtrations on these spaces. Therefore $|S_Z|$ has the same homology as $Z^{n-2}$, in particular it is $(n-3)$-acyclic.

A similar argument shows that $|(S_Z)_{<F}|$ has homology isomorphic to that of $F_{-1}$. Now, by definition of acyclic sponge, $(F,F_{-1})$ is a homology cell, in particular, $F_{-1}=\dd F$ is a homology sphere. Therefore, homology of $|(S_Z)_{<F}|$ is also isomorphic to homology of the sphere of the same dimension.

To prove the Cohen--Macaulay property of $S_Z$, we pick an arbitrary chain $F_0<\cdots<F_r$ of faces of $S_Z$ and consider the link of the simplex $\sigma=(F_0,\ldots,F_r)$ in the order complex $\ord(S_Z)$. We have
\[
\link_{\ord(S_Z)}\sigma \cong |(S_Z)_{<F_0}|\ast|(S_Z)_{(F_0,F_1)}|\ast\cdots\ast|(S_Z)_{(F_{r-1},F_r)}|\ast |(S_Z)_{>F_r}|,
\]
where $\ast$ denotes the topological join. According to the preceding discussion, the space $|(S_Z)_{<F_0}|$ has homology of a sphere of the same dimension. A poset $(S_Z)_{(F_i,F_{i+1})}$ coincides with the boolean lattice with the least and greatest elements removed, according to Remark~\ref{remDualSimplicial}. Therefore, $\ord((S_Z)_{(F_i,F_{i+1})})$ is the barycentric subdivision of the boundary of a simplex, so the geometrical realization $|(S_Z)_{(F_i,F_{i+1})}|$ is homeomorphic to a sphere.

The $(n-2-\dim F_r)$-dimensional space $|(S_Z)_{>F_r}|$ is $(n-3-\dim F_r)$-acyclic according to Lemma~\ref{lemLocHomOfC}. Indeed, the poset structure of the upper ideal $(S_Z)_{>F}$ coincides with the subposet of $C^{n-2}$ which consists of all faces that strictly contain a point $x$ of type $\dim F$. The geometrical realization $|(S_Z)_{>F}|$ is homeomorphic to $\Delta_{n-1-\dim F}^{(n-2-\dim F)}$. Hence $\link \sigma$ is $(n-3-\dim\sigma)$-acyclic.

This proves that the link of any simplex in $|S_Z|$ is acyclic below the dimension of the link. Taking into account that the space $|S_Z|$ itself is $(n-3)$-acyclic, as was proved earlier, we have shown that $S_Z$ is Cohen--Macaulay.
\end{proof}

Let $(Q^{n+1},Z^{n-2})$ be an acyclic sponge. Since $Z^{n-2}$ is a homology cell complex, the incidence numbers for pairs of cells are well defined. Let us choose orientations of all faces $F$ of $Z^{n-2}$ arbitrarily. This means that we choose a generator $o_F$ of the group $H_{\dim F}(F,\dd F)\cong \Zo$ for any $F$. For any pair of faces $F>G$, $\dim F-\dim G=1$, we consider the incidence number $\inc{F}{G}\in\Zo$ determined by the condition $\dd(o_F)=\inc{F}{G}o_G$ for the natural map $\dd\colon H_{\dim F}(F,\dd F)\to H_{\dim G}(G,\dd G)$. For any pair $F>F'$ of faces such that $\dim F-\dim F'=2$, there exist exactly two intermediate faces $F>G_1, G_2>F'$, and the following diamond relation holds:
\begin{equation}\label{eqDiamondRel}
\inc{F}{G_1}\inc{G_1}{F'}+\inc{F}{G_2}\inc{G_2}{F'}=0.
\end{equation}

In general, if $S$ is a graded poset, we use the notation $s>_it$ if $s>t$ and $\rk s-\rk t=i$. Assume that for any $s>_1t$ a number $\inc{s}{t}$ is defined, and, for any $s>_2s'$ there exist exactly two elements $t_1,t_2\in S$ between $s$ and $s'$, and the relation $\inc{s}{t_1}\inc{t_1}{s'}+\inc{s}{t_2}\inc{t_2}{s'}=0$ holds true. In this case we say that \emph{a sign convention} is set on $S$.

\begin{con}\label{conLocalCohomologyCosheaf}
Let $S$ be a poset of dimension $n-2$. We consider $S$ as a small category in a natural way: objects are the elements of $S$, there exists exactly one morphism from $s$ to $t$ if $s\geqslant t$, and no morphisms otherwise.

We consider \emph{the cosheaf $\Hh^*$ of local cohomology} of $S$, following the classical idea of \cite{Zeem}. We set $\Hh^*(s)=H^*(|S|,|S\setminus S_{\geqslant s}|)$ for $s\in S$. If $s>t$, then $S_{\geqslant s}\subset S_{\geqslant t}$, and the inclusion of pairs $(|S|,|S\setminus S_{\geqslant t}|)\to (|S|,|S\setminus S_{\geqslant s}|)$ induces the natural map $\Hh^*(s>t)\colon \Hh^*(s)\to \Hh^*(t)$.

If there is a sign convention on $S$, then we can define homology modules of the cosheaf $\Hh^*$~by
\[
H_i(S;\Hh^p)=H_i(C_*(S;\Hh^p),\dd),\quad C_i(S;\Hh^p)=\bigoplus_{\rk s=i}\Hh^p(s), \quad \dd=\bigoplus_{s>_1t}\inc{s}{t}\Hh^p(s>t).
\]

The standard argument with dihomology complex \cite[Thm.1]{Zeem} provides the spectral sequence
\begin{equation}\label{eqDihomologySequence}
E_2=H_q(S;\Hh^p)\Rightarrow H^{p-q}(|S|).
\end{equation}
\end{con}

In the following, $\relint G$ denotes the relative interior of a subset $G$. In the context of our paper, $G$ is a face of a torus action, on which the action has complexity $0$, hence $G$ is a manifold with corners. In this case $\relint G$ is the subset $G\setminus G_{-1}=G\setminus\dd G$, the topological interior of $G$.

\begin{lem}\label{lemPosetSpongeLocalIso}
Let $Z$ be an acyclic sponge, and $S_Z$ be the corresponding poset of faces. Then, for each face $F\subset Z$, there is a canonical isomorphism $H^*(Z,Z\setminus \bigsqcup_{G\geqslant F}\relint G)\cong H^*(|S_Z|,|S_Z\setminus (S_Z)_{\geqslant F}|)$.
\end{lem}

\begin{proof}
The map $f\colon Z\to |S_Z|$ constructed in the proof of Lemma~\ref{lemSZCohenMacaulay} takes the closed subset $Z\setminus \bigsqcup_{G\geqslant F}\relint G=\bigcup_{G\ngeqslant F}G$ to the closed subset $|S_Z\setminus (S_Z)_{\geqslant F}|=\bigcup_{G\ngeqslant F}|(S_Z)_{\leqslant G}|$. This map induces the isomorphism of the corresponding spectral sequences, hence it induces the natural isomorphism in relative cohomology.
\end{proof}

%
%
%
%
%

\begin{lem}\label{lemSpongeLocalIso}
Let $(M,Z)$ be an acyclic sponge. Let $x_F$ be a point lying in the relative interior of a face $F\subset Z$, and $U_{x_F}$ be a sufficiently small disk neighborhood of $x_F$ in $M$. Then there is a canonical isomorphism $H^*(Z,Z\setminus \bigsqcup_{G\geqslant F}\relint G)\cong H^*(Z,Z\setminus U_{x_F})$.
\end{lem}

\begin{proof}
Since $\bigsqcup_{G\geqslant F}\relint G$ is an open subset of the space $Z$, we can assume that $U_{x_F}\cap Z\subset \bigsqcup_{G\geqslant F}\relint G$. The intersection of each face $G\geqslant F$ with $U_{x_F}$ is a disk, and we have
\[
H^*(G\cap \overline{U}_{x_F},\dd G\cup\dd U_{x_F})\cong H^*(G,\dd G),
\]
since both are isomorphic to $\Zo$ in degree $\dim G$, and vanish otherwise. Therefore two spectral sequences
\[
(E')_1^{p,q}\cong H^{p+q}\left(Z_p,Z_{p-1}\cup (Z\setminus \bigsqcup\nolimits_{G\geqslant F}\relint G)\right)\Rightarrow H^{p+q}\left(Z,Z\setminus \bigsqcup\nolimits_{G\geqslant F}\relint G\right);
\]
\[
(E'')_1^{p,q}\cong H^{p+q}(Z_p,Z_{p-1}\cup (Z\setminus U_{x_F}))\Rightarrow H^{p+q}(Z,Z\setminus U_{x_F});
\]
are isomorphic and degenerate at the second page. This implies the statement.
\end{proof}

\section{A criterion of equivariant formality in complexity one, general case}\label{secSphereImpliesFormality}


\begin{thm}\label{thmSphereImpliesEqForm}
Let the coefficient ring $R$ be either $\Zo$ or $\Qo$. Assume that an effective smooth action of $T^{n-1}$ on a closed orientable manifold $X^{2n}$ satisfies the following properties:
\begin{enumerate}
  \item the action has nonempty finite set $X_0$ of fixed points and each face submanifold $X_F$ meets $X_0$ (condition~\eqref{eqCondJoint});
  \item the action is in general position;
  \item all stabilizers are connected;
  \item the orbit space $Q=X^{2n}/T^{n-1}$ is a homology $(n+1)$-sphere: $\Hr_i(Q)=0$ for all $i\leqslant n$, and $H_{n+1}(Q)\cong R$;
  \item For each face $F$ of $Q_{n-2}$ the identity $\Hr^i(F)=0$ holds for all $i$, and $\Hr^i(Q_{n-2})=0$ holds for all $i\leqslant n-3$.
\end{enumerate}
Then the action is equivariantly formal: $H^{\odd}(X)=0$.
\end{thm}

Conditions (4) and (5) of the theorem state that the sponge $(Q,Q_{n-2})$ of the action is acyclic.


\begin{rem}
We don't have a version of Theorem~\ref{thmSphereImpliesEqForm} for disconnected stabilizers and rational coefficients. So strictly speaking, Theorem~\ref{thmSphereImpliesEqForm} forms a criterion together with Theorem~\ref{thmEqFormImpliesSphere} only in the case of connected stabilizers.
\end{rem}

The case $n=2$ of Theorem~\ref{thmSphereImpliesEqForm} is Theorem~\ref{thmSphereImpliesEqFormSimple} proved in Section~\ref{secSphereImpliesFormalitySimple}. Indeed, for $n=2$, we have: (1)~the condition ``$X_F$ meets $X_0$'' is satisfied since $X_F$ is either a fixed point or the manifold $X$ itself; (2) there are two nonzero weights at each fixed point, so they are in $1$-general position; (3) the condition ``all stabilizers are connected'' is satisfied for semifree circle actions; condition (4) coincides with the condition 3 of Theorem~\ref{thmSphereImpliesEqFormSimple}; condition (5) is trivially satisfied for a circle action with isolated fixed points. 

We now return to Theorem~\ref{thmSphereImpliesEqForm}. By assumption, the action of $T=T^{n-1}$ on $X=X^{2n}$ is in general position, and its sponge is acyclic. Reversing the arguments of Section~\ref{secProofMain}, we see that the ABFP-sequence for $X$ is exact in degrees $\leqslant 0$ (the whole sequence vanishes in degrees $<0$, and the case of degree $0$ follows from the acyclicity of the sponge $(Q,Z)=(Q,Q_{n-2})$). If we prove the acyclicity of ABFP-sequence in positive degrees as well, then equivariant formality of $X$ will follow, according to \cite[Thm.1.1]{FP}.

From now on we assume $n\geqslant 3$. To prove the acyclicity of ABFP-sequence of $X$, we resolve it by a ``cosheaf of local ABFP-sequences''. This line of reasoning requires additional constructions introduced below.

\begin{con}\label{conABandSlice}
Let $\ABb^*(X)$ denote the non-augmented ABFP-sequence of $H^*(BT)$-modules
\begin{equation}
0\to H^*_T(X_0)\stackrel{\delta_0}{\to}
H^{*+1}_T(X_1,X_0)\stackrel{\delta_1}{\to}\cdots
\stackrel{\delta_{n-3}}{\to}H^{*+n-2}_T(X_{n-2},X_{n-3})\stackrel{\delta_{n-2}}{\to}H^{*+n-1}_T(X,X_{n-2})\to 0,
\end{equation}
that is $\ABb^i(X)$ is the graded $H^*(BT)$-module $H^*_T(X_i,X_{i-1})$ with degree shifted by $i$.

Let $x\in X$ be a point and $W_x$ be a small $T$-invariant open neighborhood of the orbit $Tx\subset X$ for which the Slice Theorem applies. This means that there exists an equivariant diffeomorphism
\begin{equation}\label{eqSliceThm}
W_x\cong T\times_{T_x} \nu_x,
\end{equation}
where $\nu_x=\tau_xX/\tau_x(Tx)$ is the normal subspace to the orbit $Tx$ (here and in the following $\tau_pM$ denotes the tangent space to a manifold $M$ at a point $p$). Let $\overline{W}_x$ and $\dd W_x$ be the closure and the boundary of $W_x$ respectively. We consider the relative ABFP-sequence of the pair $(\overline{W}_x,\dd W_x)$ (by the excision property, equivariant cohomology of this pair can be replaced by the corresponding cohomology of the pair $(X,X\setminus W_x)$):
\begin{multline}
0\to H^*_T(X_0,X_0\setminus W_x)\stackrel{\delta_0}{\to}
H^{*+1}_T(X_1,(X_1\setminus W_x)\cup X_0)\stackrel{\delta_1}{\to}\cdots
\\\stackrel{\delta_{n-3}}{\to}H^{*+n-2}_T(X_{n-2},(X_{n-2}\setminus W_x)\cup X_{n-3})\stackrel{\delta_{n-2}}{\to}H^{*+n-1}_T(X,(X\setminus W_x)\cup X_{n-2})\to 0.
\end{multline}
We denote this sequence by $\AB^*(x)$, so that $\AB^i(x)$ is the $H^*(BT)$-module $H^*_T(X_i,(X_i\setminus W_x)\cup X_{i-1})$ with grading shifted by $i$ (note that each $\AB^i(x)$ has its own internal grading). We say that $x\in X$ \emph{has type $k$} if $x\in X_k\setminus X_{k-1}$, or equivalently $\dim Tx=k$.
\end{con}

\begin{con}\label{conabmap}
Notice that for any point $x\in X$ there is a map
\begin{equation}\label{eqAbMapAnyPoint}
\ab_x\colon \AB^\ast(x)\to \ABb^*(X)
\end{equation}
induced by the inclusion of pairs $(X,\varnothing)\hookrightarrow (X,X\setminus W_x)$ (or, equivalently, by collapsing $X\setminus W_x$). The map $\ab_x$ is morphism of differential complexes of graded $H^*(BT)$-modules: in particular, it commutes with the differentials in ABFP-sequences.
\end{con}

\begin{rem}
If $x\in X$ has type $k$ in $X$, then its image in $Q$ has the same type in the sense of sponges as defined in Construction~\ref{conModelSponge}. Details can be found in~\cite{AyzCompl}. Abusing the notation we sometimes denote the image of $x$ in $Q$ with the same letter $x$.
\end{rem}

\begin{lem}\label{lemObviousVanishing}
If $x$ has type $k$, then $\AB^i(x)=0$ for $i<k$.
\end{lem}

\begin{proof}
The equivariant $i$-skeleton of $W_x$ is empty for $i<k$ since $W_x$ is small enough and does not intersect lower strata.
\end{proof}

\begin{lem}\label{lemNonObviousVanising}
If $x$ has type $k$, then the differential complex $\AB^*(x)$ is acyclic for $*>k$.
\end{lem}

\begin{proof}
The normal vector space $\nu_x$ to the orbit carries the natural action of $T_x$. $\nu_x$ can be identified with an open euclidean disk, for example the slice of $W_x$. Then we denote the corresponding closed disk and the boundary by $\overline{\nu_x}$ and $\dd\nu_x$ respectively. Let $\ABb_{T_x}^*(\overline{\nu_x},\dd\nu_x)$ be the relative ABFP-sequence of $\nu_x$ compactified at infinity (one can instead use ABFP-sequence of $\nu_x$ for cohomology with compact supports, see Remark~\ref{remCompSupp}).

We have $\nu_x\cap X_i=(\nu_x)_{i-k}$, where $(\nu_x)_{i-k}$ is the $T_x$-equivariant $(i-k)$-skeleton of the $T_x$-action on $\nu_x$. Therefore, according to \eqref{eqSliceThm}, we have
\begin{multline}\label{eqShiftedABFP}
\AB^i(x)=H^{*+i}_T(X_i,(X_i\setminus W_x)\cup X_{i-1})\cong
H_T^{*+i}(\overline{W_x}\cap X_i,\dd W_x\cup X_{i-1}) \\ \cong H^{*+i-k}_{T_x}(\overline{\nu_x}\cap X_i,\dd\nu_x\cup X_{i-1})
\cong H_{T_x}^{*+i-k}(\overline{(\nu_x)}_{i-k},\dd\nu_x\cup\overline{(\nu_x)}_{i-k-1})=\ABb_{T_x}^{i-k}(\overline{\nu}_x,\dd\nu_x).
\end{multline}
Here and in the following, we adopt the following convention to simplify the notation. If $A,B$ are subspaces of the same ambient space $E$, then $H_*(A,B)$ denotes $H_*(A,A\cap B)$ (even if $B$ is not the subspace of $A$). Isomorphisms~\eqref{eqShiftedABFP} imply that the graded module $\AB^i(x)$ is isomorphic to the graded module
$\ABb_{T_x}^{i-k}(\overline{\nu_x},\dd\nu_x)$, with degree shifted by $k$.

By collapsing $\dd \nu_x\subset \overline{\nu_x}$ to a point, we get a $2(n-k)$-sphere $S\nu_x\cong S^{2(n-k)}$ with the action of $T_x$. Therefore, $\ABb^*_{T_x}(\overline{\nu_x},\dd\nu_x)\cong \ABb^*_{T_x}(S\nu_x,\infty)$. The sphere has vanishing odd degree cohomology, hence the $T_x$-action on $S\nu_x$ is equivariantly formal. Therefore the ABFP-sequence $\ABb^*_{T_x}(S\nu_x)$ of $H^*(BT_x)$-modules is acyclic for $*>0$. The relative version $\ABb_{T_x}(S\nu_x,\infty)$ differs from $\ABb_{T_x}(S\nu_x)$ by a splitting 1-dimensional summand $H^*_{T_x}(\infty)\cong H^*(BT_x)$ at 0-th position, therefore, $\ABb_{T_x}^*(S\nu_x,\infty)$ is acyclic for $*>0$ as well. This proves acyclicity of $\AB^i(x)$ for $i>k$.
\end{proof}

\begin{con}\label{conSheafAB}
Consider a proper face $F$ of the orbit space $Q=X/T$. For all points $x$ lying over the relative interior of $F$, the complexes $\AB^*(x)$ are canonically isomorphic (here we use the fact that all stabilizers of the action are connected). We use the notation $\AB^*(F)$ for $\AB^*(x_F)$ where $x_F$ is any point lying over the relative interior of $F$. If $F>G$, then there is a natural morphism of differential complexes $\AB^*(F)\to \AB^*(G)$. Indeed, we can choose a point $x_G\in \relint G$, and a nearby point $x_F\in\relint F$ such that $W_{x_F}\subset W_{x_G}$. Therefore, inclusion of pairs $(X,X\setminus W_{x_G})\hookrightarrow (X,X\setminus W_{x_F})$ induces the maps of ABFP-sequences $\ab_{F>G}\colon \AB^*(x_F)\mapsto \AB^*(x_G)$. 
\end{con}

Since the action of $T^{n-1}$ on $X^{2n}$ is in general position, the orbit space $Q$ is a topological manifold, and the subspace $Z=Q_{n-2}$ is a sponge. If $x\in Z$ is a point of type $k$, then there is a neighborhood $U_x$ of $x$ such that the pair $(U_x,U_x\cap Z)$ is homeomorphic to $(\Ro^{n+1},C^{n-k-2}\times \Ro^k)$, where $C^{n-k-2}$ is the sponge local model, defined in Construction~\ref{conModelSponge}.

\begin{lem}\label{lemStructureOfResolutionSheaf}
If $q\leqslant n-2$, then
\begin{equation}\label{eqLocABDescription}
\AB^q(F)\cong \bigoplus_{\substack{\dim G=q,\\G\geqslant F}}H^*(BT_G).
\end{equation}
For $q=n-1$, there is an isomorphism
\begin{equation}\label{eqLocABDescriptionTop}
\AB^{n-1}(F)\cong H^*(\overline{U}_x,Z\cup \dd U_x),
\end{equation}
where the degree is shifted by $n-1$: $\AB^{n-1}(F)_k\cong H^{k+n-1}(\overline{U}_x,Z\cup \dd U_x)$.
\end{lem}

\begin{proof}
If $x$ lies in a relative interior of a face $F$, then the faces of $\overline{U_x}\cap Z$ correspond to the faces $G$ of $Z$ such that $G\geqslant F$. The relative cohomology groups $H^q(G\cap \overline{U_x},\dd G\cup \dd U_x)$ can be naturally identified with $H^q(G,\dd G)$. Indeed, both groups are concentrated in degree $q=\dim G$ where they are isomorphic to the ground ring $R$ (here we use the acyclicity of $G$ for the latter group). This is similar to the proof of Lemma~\ref{lemLocHomOfC}. Then we apply the same argument as in Section~\ref{secProofMain}, and for $x\in\relint F$ we get
\begin{multline}
\AB^q(F)=H^*_T(X_q,(X_q\setminus W_x)\cup X_{q-1})\cong\bigoplus_{\substack{\dim G=q,\\G\geqslant F}} H^*_T(X_G\cap \overline{W_x},(X_G)_{-1}\cup \dd W_x) \\
\cong\bigoplus_{\substack{\dim G=q,\\G\geqslant F}}H^*(G\cap \overline{U_x},\dd G\cup \dd U_x)\otimes H^*(BT_G)\cong \bigoplus_{\substack{\dim G=q,\\G\geqslant F}}H^*(BT_G),
\end{multline}
(the last isomorphism is due to the fact that the closure of $G\cap U_x$ is a disc).

Similarly, we have $\AB^{n-1}(F)_k\cong H^{k+n-1}(\overline{U}_x,Z\cup \dd U_x)$, since the $T$-action is free outside of~$Z$.
\end{proof}

We now look at the complexes $\AB(F)$ together with the maps $\ab_{F>G}$, defined in Construction~\ref{conSheafAB}, as a cosheaf $\AB$ on the poset $S_Q$ of proper faces of $Q$ and consider its chain complex. More precisely, for $0\leqslant p\leqslant n-2$ and $0\leqslant q\leqslant n-1$, consider the $H^*(BT)$-module
\[
\Ca^{-p,q}=\bigoplus_{F\in S_Q,\dim F=p}\AB^q(F).
\]
According to Lemma~\ref{lemObviousVanishing}, $\Ca^{-p,q}$ vanishes for $p>q$. There are two differentials on the double complex $\Ca^{*,*}$. The vertical differential
\[
d_{\ABb}\colon \Ca^{-p,q}\to \Ca^{-p,q+1}
\]
is the direct sum of differentials in ABFP-sequences of faces $F$. The horizontal differential is defined using the maps $\ab_{F>G}$ and the sign convention on $S_Q$:
\[
d_H\colon \Ca^{-p,q}\to \Ca^{-p+1,q},\qquad d_H=\bigoplus_{\substack{\dim F=p,\\\dim G=p-1,\\F>G}}\inc{F}{G}\ab_{F>G}.
\]
The diamond relation~\eqref{eqDiamondRel} implies that $d_H^2=0$.

Also there exists a natural morphism of complexes $d_H\colon C^{0,*}\to\ABb^*(X)$ given by the sum of morphisms $\ab_x$ from~\eqref{eqAbMapAnyPoint} over all fixed points $x$ (faces of rank $0$).

We add one more term $\Ca^{-(n-1),*}$ to the bigraded module $\Ca^{*,*}$, by setting
\begin{equation}\label{eqAdditionalTerm}
\Ca^{-(n-1),n-1}=\Ker(d_H\colon \Ca^{-(n-2),n-1}\to\Ca^{-(n-3),n-1}).
\end{equation}
The differential $d_H\colon \Ca^{-(n-1),n-1}\to \Ca^{-(n-2),n-1}$ is the natural inclusion of the kernel. We also set $C^{-(n-1),q}=0$ for $q<n-1$: this is motivated by Lemma~\ref{lemObviousVanishing}. By the construction of the maps $\ab_{F>G}$, the differentials $d_{\ABb}$ and $d_H$ commute. Hence the double complex with the total differential is defined:
\begin{equation}\label{eqDoubleComplex}
\left(\Ca^k=\bigoplus\nolimits_{q-p=k}\Ca^{-p,q}, d_{\Tot}=d_H+(-1)^pd_{\ABb}\right).
\end{equation}
All terms of this double complex are the graded $H^*(BT)$-modules, and the differentials are homomorphisms of $H^*(BT)$-modules. There are two natural spectral sequences of the double complex. The first sequence $E_I$ runs as follows
\begin{equation}\label{eqE1SpecSeqDefin}
(E_I)^{-p,q}_2=H^q(H^{-p}(\Ca^{*,*};d_H);d_{\ABb})\Rightarrow H^{q-p}(\Ca^*,d_{\Tot}).
\end{equation}

\begin{prop}\label{propHardCore}
The modules $H^{-p}(\Ca^{*,*};d_H)$ vanish for $p\neq 0$. The complex $(H^{0}(\Ca^{*,*}, d_H), d_{\ABb})$ is isomorphic to the complex $(\ABb^*(X),d_{\ABb})$.
\end{prop}

At first notice that there exists a homomorphism from $d_H\colon \Ca^{0,*}\cong \bigoplus_{x\in X^T}\AB^*(x)\to \ABb^*(X)$ given by the sum of $\ab_x$ over all fixed points, see Construction~\ref{conabmap}. The constructed homomorphism $d_H$ is a homomorphism of graded $H^*(BT)$-modules and it commutes with $d_{\ABb}$. Therefore, to prove Proposition~\ref{propHardCore} it is sufficient to prove the acyclicity of the augmented complexes
\begin{equation}\label{eqNewAugmented}
0\to\Ca^{-(n-1),q}\stackrel{d_H}{\to}\cdots\stackrel{d_H}{\to}\Ca^{-1,q}\stackrel{d_H}{\to}\Ca^{0,q}\stackrel{d_H}{\to}\ABb^q(X)\to 0
\end{equation}
for all gradings $q$. It will be convenient to split the proof into two lemmas: the cases $q<n-1$ and $q=n-1$ are considered separately.

\begin{lem}\label{lemHardCore1}
The sequence~\eqref{eqNewAugmented} is acyclic for $q\leqslant n-2$.
\end{lem}

\begin{proof}
If $q\leqslant n-2$, the sequence~\eqref{eqNewAugmented} takes the form
\begin{multline}
0\to\bigoplus_{\substack{F\in S_Q,\\\dim F=n-2}}\AB^q(F)\to\bigoplus_{\substack{F\in S_Q,\\\dim F=n-3}}\AB^q(F)\to\cdots \\\cdots\to
\bigoplus_{\substack{F\in S_Q,\\\dim F=1}}\AB^q(F)\to \bigoplus_{\substack{F\in S_Q,\\\dim F=0}}\AB^q(F)\to \ABb^q(X)\to 0
\end{multline}
According to \eqref{eqLocABDescription} (and~\eqref{eqTech1} in the last term), this writes as
\begin{multline}\label{eqExactSeqOfProperFace}
0\to\bigoplus_{\dim F=n-2}\bigoplus_{\substack{\dim G=q,\\G\geqslant F}}H^*(BT_G)\to\cdots \to\bigoplus_{\dim F=1}\bigoplus_{\substack{\dim G=q,\\G\geqslant F}}H^*(BT_G)\to \\ \to \bigoplus_{\dim F=0}\bigoplus_{\substack{\dim G=q,\\G\geqslant F}}H^*(BT_G)\to \bigoplus_{\dim G=q} (H^*(BT_G)\otimes H^*(G,G_{-1}))_{*+q}\to 0.
\end{multline}
Since $\dim G=q\leqslant n-2$, the assumption of Theorem~\ref{thmSphereImpliesEqForm} implies that $H^*(G,G_{-1})$ is only nontrivial in degree $q$, where it has rank $1$, therefore the last term of~\eqref{eqExactSeqOfProperFace} can be identified with $H^*(BT_G)$.

Changing the summation order in~\eqref{eqExactSeqOfProperFace} we see that it is the direct sum over all faces $G$ of dimension $q$ of the following complexes
\[
0\to \bigoplus_{\substack{\dim F=n-2\\F\leqslant G}}H^*(BT_G)\to\cdots\to \bigoplus_{\substack{\dim F=1\\F\leqslant G}}H^*(BT_G)\to \bigoplus_{\substack{\dim F=0\\F\leqslant G}}H^*(BT_G)\to H^*(BT_G)\to 0.
\]
This complex coincides with the reduced cellular chain complex of the face $G$ with coefficients in $H^*(BT_G)$. Since $G$ is acyclic by assumption, this complex is acyclic.
\end{proof}

\begin{lem}\label{lemHardCore1}
The sequence~\eqref{eqNewAugmented} is acyclic for $q=n-1$.
\end{lem}

\begin{proof}
This case is more complicated. The complex~\eqref{eqNewAugmented} is the complex of graded $H^*(BT)$-modules. According to \eqref{eqLocABDescriptionTop} and \eqref{eqTech2}, its component of internal degree $k$ has the form
\begin{multline}\label{eqTopRowStrange}
0\to(\Ca^{-(n-1),n-1})_k\to \bigoplus_{\dim F=n-2}H^k(\overline{U}_{x_F}, Z\cup \dd U_{x_F})\to\cdots\\ \cdots\to \bigoplus_{\dim F=1}H^k(\overline{U}_{x_F},Z\cup \dd U_{x_F})\to \bigoplus_{\dim F=0}H^k(\overline{U}_{x_F},Z\cup \dd U_{x_F})\to H^{k}(Q,Z)\to 0,
\end{multline}
where $x_F$ is any point in the interior of a face $F\subset Q$, $U_{x_F}$ is a small disk neighborhood of $x_F$, and the rightmost augmentation homomorphism is induced by the inclusion $(Q,Z)\hookrightarrow (Q,Z\cup (Q\setminus U_{x_F}))$. Notice, that all modules appearing in~\eqref{eqTopRowStrange} are $H^*(BT)$-modules and the differentials are homomorphisms of $H^*(BT)$-modules. However, this structure does not carry any important information.

\textbf{Case 1.} Let us first assume $k\leqslant n$, and consider a summand of any term of~\eqref{eqTopRowStrange} except the augmentation term. We have
\begin{multline}\label{eqSeriesOfIso}
H^{k}(\overline{U}_{x_F},Z\cup \dd U_{x_F})\stackrel{(1)}{\cong} H^k(Q,Z\cup (Q\setminus U_{x_F}))\stackrel{(2)}{\cong} H^{k-1}(Z,Z\setminus U_{x_F}) \stackrel{(3)}{\cong} \\
H^{k-1}(Z, Z\setminus \bigsqcup_{G\geqslant F}\relint G)\stackrel{(4)}{\cong} H^{k-1}(|S_Z|,|S_Z\setminus(S_Z)_{\geqslant F}|)\stackrel{(5)}{=} \Hh^{k-1}(F).
\end{multline}

Isomorphism (1) is the excision isomorphism for the set $Q\setminus \overline{U}_{x_F}$. Isomorphism (2) follows from the cohomology long exact sequence of the pair $(Q,Z)$ relative to $Q\setminus U_{x_F}$:
\[
\cdots\to H^{k-1}(Q,Q\setminus U_{x_F})\to H^{k-1}(Z,Z\setminus U_{x_F})\to H^k(Q,Z\cup(Q\setminus U_{x_F}))\to H^k(Q,Q\setminus U_{x_F})\to\cdots
\]
and the fact that $H^*(Q,Q\setminus U_{x_F})\cong H^*(S^{n+1})$ vanishes in degrees $<n+1$. Isomorphism (3) of \eqref{eqSeriesOfIso} is proved in Lemma~\ref{lemSpongeLocalIso}. Isomorphism (4) is proved in Lemma~\ref{lemPosetSpongeLocalIso}. Equality (5) is the definition of the cosheaf $\Hh^*$, see Construction~\ref{conLocalCohomologyCosheaf}.

The natural isomorphisms of \eqref{eqSeriesOfIso} imply that, for $k\leqslant n$, the differential complex
\begin{multline}\label{eqTopRowStrangeCut}
0\to \bigoplus_{\dim F=n-2}H^k(\overline{U}_{x_F}, Z\cup \dd U_{x_F})\to\cdots\\\cdots\to \bigoplus_{\dim F=1}H^k(\overline{U}_{x_F},Z\cup \dd U_{x_F})\to \bigoplus_{\dim F=0}H^k(\overline{U}_{x_F},Z\cup \dd U_{x_F})\to 0,
\end{multline}
(the complex~\eqref{eqTopRowStrange} without left and right augmentations) is isomorphic to $C_*(S_Z,\Hh^{k-1})$, the chain complex of the cosheaf of local cohomology. The poset $S_Z$ is Cohen--Macaulay by Lemma~\ref{lemSZCohenMacaulay}, hence the cosheaf $\Hh^*$ is concentrated in degree $n-2$ (which means that $\Hh^*(F)=0$ for any $F\in S_Z$ and $\ast\neq n-2$). Therefore, the dihomology spectral sequence~\eqref{eqDihomologySequence} collapses at the second page. This fact implies the isomorphism
\begin{equation}\label{eqNewIsom}
H_r(S_Z;\Hh^{n-2})\cong H^{n-2-r}(S_Z).
\end{equation}
Since $S_Z$ is $(n-3)$-acyclic, the module $H^{n-2-r}(S_Z)$ is nontrivial only for $r=n-2$ and $r=0$.

The homology of the complex \eqref{eqTopRowStrangeCut} in the leftmost position are killed by the additional term $C^{-(n-1),n-1}$ defined by~\eqref{eqAdditionalTerm}.

The homology module of \eqref{eqTopRowStrange} at the rightmost position is isomorphic to
\begin{equation}\label{eqTopMiddle}
H_0(S_Z;\Hh^{n-2})\cong H^{n-2}(S_Z)\cong H^{n-2}(Z)\cong H^{n-1}(Q,Z),
\end{equation}
where the last isomorphism follows from the long exact sequence in cohomology of the pair $(Q,Z)$ and the assumption of the theorem, which states that $Q$ is a homology $(n+1)$-sphere.
The fact that the isomorphism of~\ref{eqTopMiddle} is actually induced by the augmentation homomorphism $d_H\colon \Ca^{0,n-1}\to \ABb^{n-1}(X)\cong H^k(Q,Z)$ introduced earlier is explained as follows.

Recall that there exists a homomorphism $\ab_{x_F}\colon H^*(\overline{U}_{x_F},Z\cup \dd U_{x_F})\to H^*(Q,Z)$ induced by collapsing $Q\setminus U_x$, and these homomorphisms commute with the defining maps $\ab_{F>G}$ of the sheaf (since the latter are also induced by collapses). Via the sequence of homeomorphisms~\eqref{eqSeriesOfIso}, the maps $\ab_{x_F}$ provide the morphism $g$ from the cosheaf $\Hh^{n-2}(F)$ to the constant cosheaf on $S_Z$ taking value $H^{n-1}(Q,Z)$. By functoriality of cosheaf homology, we get the induced homomorphism $g_*\colon H_0(S_Z;\Hh^{n-2})\to H_0(S_Z;H^{n-1}(Q,Z))$. The target module $H_0(S_Z;H^{n-1}(Q,Z))$ is naturally isomorphic to $H^{n-1}(Q,Z)$ since $S_Z$ is connected (here we use the assumption $n\geqslant 3$). Now it remains to prove that $g$ fits into the commutative diagram
\begin{equation}\label{eqSheafDiagr}
\xymatrix{
H_0(S_Z,\Hh^{n-2})\ar@{<->}[r]^(.55){\cong} \ar@{->}[rd]^{g_*}& H^{n-2}(|S_Z|) \ar@{->}[r]^{\cong} &H^{n-2}(Z) \ar@{=}[d]\\
& H_0(S_Z,H^{n-1}(Q,Z))\ar@{->}[r]^(.6){\cong} & H^{n-2}(Z)
}
\end{equation}
where the first isomorphism in the top row is given by Zeeman's dihomology as explained above, and the second follows from the fact that $Z$ and $S_Z$ are homologous (see the proof of Lemma~\ref{lemSZCohenMacaulay}). To prove this, we first give an alternative description for the values of the cosheaf $\Hh$. We take a point $x=x_F$ lying in the relative interior of a face $F$, and a small disk neighborhood $U_x$. Let $\ca{S}_x^{n+1}$ denote the $(n+1)$-dimensional sphere obtained by collapsing $Q\setminus U_x$ to a point in the orbit space: $\ca{S}_x=Q/(Q\setminus U_x)\cong \overline{U_x}/\dd U_x$. Let $Z_x$ denote the image of $Z$ under this collapse: $Z_x=Z/(Q\setminus U_x)$. Also set $|S_Z|_x=|S_Z|/|S_Z\setminus(S_Z)_{\geqslant F}|$, this is the combinatorial counterpart of $Z_x$. Then we have the following commutative diagram
\begin{equation}\label{eqSheafValDiagr}
\xymatrix{
\Hr^{n-2}(|S_Z|_x)\ar@{->}[r]^{\cong}\ar@{->}[d]& \Hr^{n-2}(Z_x)\ar@{->}[r]^(.4){\cong}\ar@{->}[d] & H^{n-1}(\ca{S}_x^{n+1}, Z_x) \ar@{->}[d]\\
\Hr^{n-2}(|S_Z|)\ar@{->}[r]^{\cong}& \Hr^{n-2}(Z)\ar@{->}[r]^(.4){\cong} & H^{n-1}(Q, Z).
}
\end{equation}
The first isomorphism in the top row is already explained in~\eqref{eqSeriesOfIso}, and the second isomorphism follows from the long exact sequence of the pair $(\ca{S}_x^{n+1}, Z_x)$ since $\ca{S}_x^{n+1}$ is a sphere. The isomorphisms of the bottom row are explained similarly. The right square in~\eqref{eqSheafValDiagr} is commutative since both vertical maps are induced by collapsing $Q\setminus U_x$. The left square is commutative, since collapsing of $Q\setminus U_x$ and $|S_Z\setminus(S_Z)_{\geqslant F}|$ is compatible with the homology equivalence from $Z$ to $|S_Z|$ constructed in the proof of Lemma~\ref{lemSZCohenMacaulay}.

Commutative diagram~\eqref{eqSheafValDiagr} proves that the morphism $g$ of cosheaves coincides, up to isomorphism, with the natural morphism $\Hh^{n-2}(F)=H^*(|S_Z|,|S_Z\setminus (S_Z)_{\geqslant F}|) \to \Hr^*(|S_Z|)$ from a local to the global cohomology of the poset $S_Z$. The fact that the latter homomorphisms assemble to the Zeeman's isomorphism $H_0(S_Z;\Hh^{n-2})\stackrel{\cong}{\to} H^{n-2}(|S_Z|)$, becomes a tautological consequence of Zeeman's constructions~\cite{Zeem}.

\textbf{Case 2.} Now we study the differential complex~\eqref{eqTopRowStrange} for $k=n+1$. Since $Z$ has dimension $n-2$, we can drop this space from the second position at all relative cohomology groups. Therefore the degree $n+1$ part of~\eqref{eqTopRowStrange} takes the form
\begin{multline}\label{eqTopRowStrangeTop}
0\to(\Ca^{-(n-1),n-1})_{n+1}\to \bigoplus_{\dim F=n-2}H^{n+1}(\overline{U}_{x_F}, \dd U_{x_F})\to\cdots\\\cdots\to \bigoplus_{\dim F=1}H^{n+1}(\overline{U}_{x_F},\dd U_{x_F})\to \bigoplus_{\dim F=0}H^{n+1}(\overline{U}_{x_F},\dd U_{x_F})\to H^{n+1}(Q)\to 0.
\end{multline}
Since $\overline{U}_{x_F}$ is an $(n+1)$-ball, this sequence writes as
\begin{equation}\label{eqTopRowStrangeTopZ}
0\to(\Ca^{-(n-1),n-1})_{n+1}\to \bigoplus_{\dim F=n-2}\Zo\to\cdots\to\bigoplus_{\dim F=1}\Zo\to \bigoplus_{\dim F=0} \Zo\to\Zo\to 0.
\end{equation}
Removing the augmentation from the left, we get
\begin{equation}\label{eqTopRowStrangeTopZcut}
0\to\bigoplus_{\dim F=n-2}\Zo\to\cdots\to \bigoplus_{\dim F=1}\Zo\to \bigoplus_{\dim F=0}\Zo\to \Zo\to 0.
\end{equation}
This is the reduced chain complex of the homology cell complex $Z$, and its homology is concentrated in the leftmost position. The additional term $(\Ca^{-(n-1),n-1})_{n+1}$ of the differential complex~\eqref{eqTopRowStrangeTopZ} kills the homology at the leftmost position, because it was defined by \eqref{eqAdditionalTerm} to do so. Therefore the complex~\eqref{eqTopRowStrangeTop} is acyclic. This completes the proof of the lemma.
\end{proof}

Proposition~\ref{propHardCore} shows that the spectral sequence $E_I$ given by~\eqref{eqE1SpecSeqDefin} degenerates at the second page. We have
\begin{equation}\label{eqE1SpecSecConclusion}
(E_I)_\infty^{0,k}\cong H^k(\ABb^*(X);d_{\ABb})\mbox{ is an associated graded module for }H^k(\Ca^*; d_{\Tot}).
\end{equation}

There exists another cohomological spectral sequence, which computes Atiyah--Bredon cohomology first, then computes vertical cohomology:
\begin{equation}\label{eqE2SpecSeqDefin}
(E_{II})^{-p,q}_2=H^{-p}(H^q(\Ca^{*,*};d_{\ABb});d_H)\Rightarrow H^{q-p}(\Ca^*,d_{\Tot}).
\end{equation}

\begin{lem}\label{lemDiagonalConcentration}
The cohomology $H^q(\Ca^{-p,*};d_{\ABb})$ vanishes for $q\neq p$.
\end{lem}

\begin{proof}
For $0\leqslant p\leqslant n-2$ we have $\Ca^{-p,q}=\bigoplus_{\dim F=p}\AB^q(F)$ by definition. The cohomology of $(\AB^q(F),d_{\ABb})$ vanishes for $q\neq\dim F$ by Lemmas~\ref{lemObviousVanishing} and \ref{lemNonObviousVanising}. If $p=n-1$, the additional term $C^{-(n-1),q}$ is concentrated in degree $q=n-1$ by construction.
\end{proof}

Lemma~\ref{lemDiagonalConcentration} shows that the spectral sequence $E_{II}$ given by~\eqref{eqE2SpecSeqDefin} satisfies
\[
(E_{II})^{-p,q}_1=0\mbox{ for } p\neq q.
\]
Therefore, $E_{II}$ degenerates at the first page, and we have
\begin{equation}\label{eqE2SpecSecConclusion}
H^k(\Ca^*; d_{\Tot})=0\mbox{ for } k\neq 0.
\end{equation}

\begin{proof}[Proof of Theorem~\ref{thmSphereImpliesEqForm}]
Combining \eqref{eqE1SpecSecConclusion} with \eqref{eqE2SpecSecConclusion}, we see that ABFP-sequence $\ABb^*(X)$ of the space $X$ is acyclic in degrees $\neq0$. The differential complex $\ABb^*(X)$ is the first page of the spectral sequence
\[
(E_T)^{p,q}_1\cong H^{p+q}_T(X_p,X_{p-1})\Rightarrow H^{p+q}_T(X).
\]
The acyclicity of $\ABb^*(X)$ implies that $H^0(\ABb^*(X),d_{\ABb})\cong H^*_T(X)$. Therefore, the augmented ABFP-sequence
\begin{multline}
0\to H^*_T(X)\stackrel{i^*}{\to} H^*_T(X_0)\to
H^{*+1}_T(X_1,X_0)\to\cdots\\\cdots
\to H^{*+n-2}_T(X_{n-2},X_{n-3})\to H^{*+n-1}_T(X,X_{n-2})\to 0,
\end{multline}
is exact. According to \cite[Thm.1.1]{FP}, this condition implies that $X$ is equivariantly formal. This completes the proof of Theorem~\ref{thmSphereImpliesEqForm}.
\end{proof}

\section{Betti numbers in case of complexity one}\label{secBetti}

\begin{defin}\label{definFvector}
Consider a sponge $Z=Z^{n-2}$. Let $f_i=f_i(Z)$ be the number of $i$-dimensional faces of $Z$, for $i=0,1,\ldots,n-2$, and $b=b(Z)$ denotes the Betti number $b=\rk \Hr_{n-2}(Z)$ (which is the only nonzero reduced Betti number of $Z$, according to the acyclicity condition). The pair $((f_0,\ldots,f_{n-2}),b)$ will be called \emph{the extended f-vector} of the sponge $Z$.
\end{defin}

\begin{rem}
If $Z$ is an acyclic sponge, then we have $f_0-f_1+\cdots+(-1)^{n-2}f_{n-2}=1+(-1)^{n-2}b$, since both numbers are equal to the Euler characteristic of $Z$. Therefore, for acyclic sponges, the $b$-number is expressed in terms of f-numbers:
\[
b=f_{n-2}-f_{n-1}+\cdots+(-1)^{n-2}f_0+(-1)^{n-1}f_{-1},
\]
where we set $f_{-1}=1$ by definition.
\end{rem}

In the following $\Hilb(A^*,t)=\sum_{i=-\infty}^{+\infty}(\rk A_i) t^i$ denotes the Hilbert--Poincare series of a $\Zo$-graded $R$-module.

\begin{prop}
The Betti numbers of an orientable manifold $X$ with equivariantly formal action of complexity one in general position can be expressed from the extended $f$-vector of its sponge by the formula
\begin{equation}\label{eqBettiSimplerFormula}
\Hilb(H^*(X);t)=\sum_{i=0}^{n-2}(-1)^if_i(1-t^2)^{i}+(b+t^2)(t^2-1)^{n-1}.
\end{equation}
\end{prop}

\begin{proof}
Since $X$ is equivariantly formal, its ABFP-sequence is exact. Therefore, taking Euler characteristic of the ABFP-sequence in each degree, we get
\begin{equation}\label{eqEquivEulChar}
\Hilb(H^*_T(X);t)=\sum_{i=0}^{n-1}(-1)^i\Hilb(\ABb^i(X);t)=\sum_{i=0}^{n-2}(-1)^i\dfrac{f_i}{(1-t^2)^{n-1-i}}+(-1)^{n-1}(b+t^2),
\end{equation}
where the last identity is due to relations \eqref{eqTech1} and \eqref{eqTech2}.

Equivariant formality of $X$ implies that $H^*_T(X)$ is a free $H^*(BT)$-module, and $H^*(X)\cong H^*_T(X)\otimes_{H^*(BT)}R$. For the Hilbert--Poincare series this implies
\begin{equation}
\Hilb(H^*(X);t)=\dfrac{\Hilb(H^*_T(X);t)}{\Hilb(H^*(BT);t)}=\Hilb(H^*_T(X);t)\cdot(1-t^2)^{n-1}.
\end{equation}

Multiplying both sides of~\eqref{eqEquivEulChar} by $(1-t^2)^{n-1}$ we get the required formula for the ordinary cohomology.
\end{proof}

\begin{rem}
The homological arguments of Section~\ref{secSphereImpliesFormality} can be indirectly checked by examining Hilbert--Poincare series of all modules and sheaves appearing in the proof. As mentioned in the last paragraph of Section~\ref{secSphereImpliesFormality}, $H^*_T(X)\cong H^0(\ABb^*(X))$. Isomorphisms~\ref{eqE1SpecSecConclusion} and \ref{eqE2SpecSecConclusion} imply that $(E_{II})_\infty^{*,*}=\bigoplus_{p=0}^{n-1}(E_{II})^{-p,p}_1$ is an associated module for $H^*_T(X)$. Hilbert--Poincare series of all terms of $(E_{II})_\infty^{*,*}$ can be expressed via the extended $f$-vector, which in the end leads to the formula
\begin{equation}\label{eqHilbEquiv}
\Hilb(H^*_T(X);t)=\sum_{F\in S_Z}\dfrac{t^{2n-2\dim F}}{(1-t^2)^{n-1-\dim F}}+(1+bt^2)=\sum_{i=0}^{n-2}\dfrac{f_it^{2n-2i}}{(1-t^2)^{n-1-i}}+(1+bt^2),
\end{equation}
and, consecutively,
\begin{equation}\label{eqBettiHarderFormula}
\Hilb(H^*(X);t)=\sum_{i=0}^{n-2}f_it^{2n-2i}(1-t^2)^{i}+(1+bt^2)(1-t^2)^{n-1}.
\end{equation}
One can notice that~\eqref{eqBettiHarderFormula} differs from~\eqref{eqBettiSimplerFormula}. However, one formula transforms into another by applying Poincare duality on $X$.
\end{rem}

\begin{ex}
Let $M^{2n}$ be a manifold with an equivariantly formal $T^n$-action. According to Lemma~\ref{lemMasPan}, the orbit space $P=M/T^n$ is a face acyclic manifold with corners, and it has the $f$-vector $(f_0,f_1,\ldots,f_{n-1},f_n=1)$, where $f_i$ is the number of $i$-dimensional faces of $P$. As was mentioned in Example~\ref{exExamplesSponges}, whenever we have an induced action of a subtorus $T^{n-1}\subset T^n$ on $M^{2n}$ which is in general position, the sponge $Z$ of this action is the $(n-2)$-skeleton of $P$. The extended $f$-vector of this sponge is equal to $((f_0,f_1,\ldots,f_{n-2}),b)$. Using Euler characteristic, it is easily seen that $b=f_{n-1}-1$. In this case, formula~\eqref{eqBettiHarderFormula} coincides with the definition of $h$-numbers of the polytope $P$ (see Remark~\ref{remHnumbersOfPoly} below).
\end{ex}

\begin{ex}
For the action of $T^3$ on the Grassmann manifold $G_{4,2}$, the sponge consists of all proper faces and three equatorial squares of an octahedron (see the top left part of Fig.\ref{figSpongeExamples} in Section~\ref{secSpongesHomology}). Its extended $f$-vector is $((6,12,11),4)$. We have
\[
\Hilb(H^*(G_{4,2});t)=6+12(t^2-1)+11(t^2-1)^2+(4+t^2)(t^2-1)^3=1+t^2+2t^4+t^6+t^8.
\]
This coincides with the well known answer for Betti numbers of the Grassmann manifold.
\end{ex}

\begin{ex}
For any action of $T^2$ on a $6$-dimensional manifold in general position, the sponge coincides with the GKM-graph of the action. In particular, for the action of $T^2$ on the full flag manifold $F_3$, the sponge is the complete bipartite graph $K_{3,3}$, see the top right part of Fig.\ref{figSpongeExamples}. The extended $f$-vector is $((6,9),4)$, and we obtain
\[
\Hilb(H^*(F_3);t)=6+9t(t^2-1)+(4+t^2)(t^2-1)^2=1+2t^2+2t^4+t^6.
\]
\end{ex}

\begin{ex}
The sponge of the $T^3$-action on the quaternionic projective plane $\HP^2$ was described in detail in \cite{AyzHP}, see the bottom of Fig.\ref{figSpongeExamples} (triangles correspond to torus-invariant manifolds isomorphic to $\CP^2$ and biangles correspond to torus invariant manifolds isomorphic to $\HP^1$, they are glued together along edges as marked on the figure). The extended f-vector of this sponge is equal to $((3,6,7),3)$, and we obtain
\[
\Hilb(H^*(\HP^2);t)=3+6(t^2-1)+7(t^2-1)^2+(3+t^2)(t^2-1)^3=1+t^4+t^8.
\]
\end{ex}

We finish with the following natural definition and the question.

\begin{defin}
Let $Z=Z^{n-2}$ be an acyclic sponge, and $((f_0,\ldots,f_{n-2}),b)$ be the extended $f$-vector of $Z$. The coefficients $(h_0,h_1,\ldots,h_n)$ of the polynomial
\[
\sum_{i=0}^{n-2}f_it^{2n-2i}(1-t^2)^{i}+(1+bt^2)(1-t^2)^{n-1}=h_0+h_1t^2+\cdots+h_nt^{2n}
\]
are called \emph{the $h$-numbers} of $Z$.
\end{defin}

Equivalently, $\sum_{i=0}^{n-2}(-1)^if_i(1-t^2)^{i}+(-1)^{n-1}(b+t^2)(1-t^2)^{n-1}=h_n+h_{n-1}t^2+\cdots+h_0t^{2n}$.

\begin{probl}[Dehn--Sommerville relations]\label{problDehnSomm}
Is it true that $h_i=h_{n-i}$ for any acyclic sponge $Z$?
\end{probl}

\begin{probl}[Nonnegativity]\label{problLowerBound}
Is it true that $h_i\geqslant 0$ for any acyclic sponge $Z$?
\end{probl}

Certainly, if $Z$ is a sponge of some equivariantly formal torus action of complexity one in general position, both questions answer in positive, as follows from nonnegativity of Betti numbers and Poincare duality on the corresponding $T^{n-1}$-manifold.

\begin{rem}\label{remHnumbersOfPoly}
The $h$-vector $(h_0,\ldots,h_n)$ of a simple $n$-dimensional polytope $P$ is defined by the formula
\[
\sum_{i=0}^{n}h_it^{2i}=\sum_{i=0}^{n}f_i(1-t^2)^it^{2n-2i},
\]
where $f_i$ is the number of $i$-dimensional faces of $P$, see e.g.~\cite[Ch.1.3]{BPnew}. Similarly, we can define the $h$-vector of an $n$-dimensional manifold with corners $P$, if all faces of $P$ (including $P$ itself) are acyclic. If $M^{2n}$ is an equivariantly formal manifold with the complexity zero action of a torus $T^n$ such that $M^{2n}/T^n\cong P$, then $\rk H^{2j}(M^{2n})=h_j$ (see \cite{DJ} for simple polytopes and \cite{MasPan} for face acyclic manifolds with corners). Hence, when $P$ is the orbit space of some action the nonnegativity $h_j\geqslant 0$ obviously follows, and Dehn--Sommerville relations $h_j=h_{n-j}$ are a consequence of Poincare duality. However, Dehn--Sommerville relations and the nonnegativity hold for any face-acyclic manifold with corners as follows from the Gorenstein property of the face ring $\ko[S_P]$ corresponding to the simplicial poset $S_P$ dual to $P$.

We suppose that there should exist a theory of ``sponge algebras'' which is parallel to the theory of face rings for simple polytopes. This theory, if exists, should answer Problems~\ref{problDehnSomm} and \ref{problLowerBound}, and give a description of equivariant cohomology rings for equivariantly formal actions of complexity one in general position.
\end{rem}

\textbf{Acknowledgements.} The authors thank the anonymous referees for the numerous valuable comments on the first and second versions of the paper, especially for paying our attention that orientability assumption is required in all statements of the paper and that the reference to Poincar\'{e} conjecture in dimension 4 was missing. We also appreciate the advice of the referee to rewrite some arguments in the language of homology with closed supports.

\textbf{Funding.} The article was prepared within the framework of the HSE University Basic Research Program.

\textbf{Conflict of interest.} The authors declare that they have no conflict of interest.

\end{document}